\documentclass[draft,reqno]{amsproc}
\usepackage{amsmath}
\usepackage{amsfonts}
\usepackage{mathrsfs}
\usepackage[dvips]{graphicx}
\usepackage{amssymb}
\usepackage{amsbsy}
\usepackage{amsthm}
\usepackage{graphics}
\usepackage{color}
\usepackage{textcomp}
\usepackage{verbatim}

\usepackage{tikz}
\usetikzlibrary{arrows.meta,calc}
\makeatletter
\@namedef{subjclassname@2010}{%
\textup{2010} Mathematics Subject
Classification} \makeatother

\DeclareMathOperator{\supp}{supp}

\DeclareMathOperator{\dfi}{\mathscr{D}}

\numberwithin{equation}{section}

\newtheorem{theorem}{Theorem}[section]
\newtheorem*{theorem*}{Theorem}
\newtheorem{corollary}[theorem]{Corollary}
\newtheorem{prop}[theorem]{Proposition}

\newtheorem*{prob*}{Problem}

\newtheorem{lemma}[theorem]{Lemma}

\newtheorem*{manualtheorem}{Theorem~$2.2'$}

\theoremstyle{remark}
\newtheorem{rem}[theorem]{Remark}
\newtheorem{df}[theorem]{Definition}

\newcommand*{\ascr}{\mathscr{A}}

\newcommand*{\hh}{\mathcal{H}}
\newcommand*{\is}[2]{\langle#1,#2\rangle}
\newcommand*{\kk}{\mathcal{K}}
\newcommand*{\lcal}{\mathcal{L}}
\newcommand*{\lambdab}{\boldsymbol{\lambda}}
\newcommand*{\natu}{\mathbb{N}}
\newcommand*{\mcal}{\mathcal{M}}

\newcommand*{\nbold}{\boldsymbol{N}}
\newcommand*{\tbold}{\boldsymbol{T}}
\newcommand*{\zbold}{\boldsymbol{z}}

\newcommand*{\nul}{\mathscr{N}}
\newcommand*{\comp}{\mathbb{C}}
\newcommand*{\borel}{\mathfrak{B}}

\newcommand*{\cbb}{\comp}

\newcommand*{\D}{\mathrm{d\hspace{.1ex}}}

\theoremstyle{definition}
\newtheorem{exa}[theorem]{Example}

\newcommand*{\Le}{\leqslant}
\newcommand*{\Ge}{\geqslant}
\newcommand*{\ran}{\mathscr{R}}

\newcommand*{\ogr}{\boldsymbol{B}}
\newcommand*{\rbb}{\mathbb{R}}
\newcommand*{\zbb}{\mathbb{Z}}
\newcommand*{\tbb}{\mathbb{T}}

\begin{document}

   \title[Hyperrigidity II:
$R$-dilations, ideals and decompositions]
{Hyperrigidity II: $R$-dilations, ideals and
decompositions}

   \author[P. Pietrzycki and J. Stochel]{Pawe{\l} Pietrzycki and  Jan Stochel}



   \address{Wydzia{\l} Matematyki i Informatyki, Uniwersytet
Jagiello\'{n}ski, ul. {\L}ojasiewicza 6,
PL-30348 Krak\'{o}w}

   \email{pawel.pietrzycki@im.uj.edu.pl}

   \address{Wydzia{\l} Matematyki i Informatyki, Uniwersytet
Jagiello\'{n}ski, ul. {\L}ojasiewicza 6,
PL-30348 Krak\'{o}w}

   \email{jan.stochel@im.uj.edu.pl}

   \thanks{The research of both
authors was supported by the National
Science Center (NCN) Grant OPUS No.\
DEC-2024/55/B/ST1/00280.}

   \begin{abstract}
We investigate the hyperrigidity of subsets of unital
$C^*$-algebras annihilated by states (or, more
generally, by completely positive maps). This is
closely related to the concept of rigidity at $0$
introduced by G. Salomon, who studied hyperrigid
subsets of Cuntz and Cuntz-Krieger algebras. The
absence of the unit in a hyperrigid set allows for the
existence of $R$-dilations with non-isometric $R$. The
existence of such an $R$-dilation forces the state
annihilating the hyperrigid set to be a character.
Using a dilation-theoretic approach, we provide
multiple equivalent criteria for hyperrigidity
involving intertwining relations for representations,
valid in both commutative and noncommutative settings.
We develop structural models for such dilations via
orthogonal decompositions into two or three
components, determined by defect operators and
generalized joint eigenspaces associated with
underlying representations.
   \end{abstract}

   \maketitle

   \section{Introduction}
Motivated both by the fundamental role of the
classical Choquet boundary in classical approximation
theory and by the importance of approximation in the
contemporary theory of operator algebras, Arveson
introduced hyperrigidity as a form of approximation
that captures many important operator-algebraic
phenomena. Below, we write $\natu$ for the set of all
positive integers. $\ogr(\hh)$ denotes the
$C^*$-algebra of all bounded linear operators on a
Hilbert space $\hh$, and $I=I_{\hh}$ stands for the
identity operator on $\hh$. Throughout, all Hilbert
spaces are assumed to be complex. By a representation
we mean a unital $*$-homomorphism into some
$\ogr(\hh)$.
   \begin{df} \label{dyfnh}
A finite or countably infinite set $G$
of generators of a unital $C^*$-algebra
$\ascr$ is said to be
\textit{hyperrigid} if for any faithful
representation $\pi\colon \ascr\to
\ogr(\hh)$ on a separable Hilbert space
$\hh$ and for any sequence
$\varPhi_k\colon \ogr(\hh)\to
\ogr(\hh)$ ($k\in \natu$) of unital
completely positive (UCP) maps,
   \begin{align*}
\lim_{k\to\infty}\|\varPhi_k(\pi(g))-\pi(g)\|=0
\; \forall g\in G \implies
\lim_{k\to\infty}\|\varPhi_k(\pi(a))-\pi(a)\|=0
\; \forall a\in \ascr.
   \end{align*}
   \end{df}
Arveson conjectured that hyperrigidity
is equivalent to the noncommutative
Choquet boundary of $G$ being as large
as possible, meaning that every
irreducible representation of $C^*(G)$
is a boundary representation for $G$.
This statement, now known as {\em
Arveson's hyperrigidity conjecture}
(see \cite[Conjecture~4.3]{Arv11}), has
been verified for certain classes of
$C^*$-algebras
\cite{Arv11,Kle14,Brown16,CH18,Sal19,Ka-Sh19,Sh20,P-S-24,Sch24}.
In full generality, however, the
conjecture fails: a counterexample was
recently provided by B. Bilich and A.
Dor-On \cite{Bi-Dor24} (see also
\cite{Bi24}; for a finite-dimensional
example see \cite{Sch25}).
Nevertheless, the conjecture remains
open for commutative $C^*$-algebras. In
recent years, it has attracted
considerable attention across various
areas of operator algebras and operator
theory
\cite{KS15,Kle14,Clo18,Clo18b,CH18,Har19,Cl-Ti21,DK21,CH-Th24,Thom24,P-S25}.

In \cite[Corollary~3.4]{Arv11}, Arveson proved that
for every integer $n \in \natu$, the set of $n$
isometries $G_n = \{V_1, \ldots, V_n\}$ generating the
Cuntz algebra $\mathcal{O}_n$ is hyperrigid. In
contrast, Muhly and Solel showed in \cite{MuhSol98}
that the infinite set of isometries $G_\infty = \{V_1,
V_2, \ldots\}$ generating the Cuntz algebra
$\mathcal{O}_\infty$ is not hyperrigid. Salomon proved
that for every $n\in \natu\cup \{\infty\}$, there
exists a state on $\mathcal{O}_n$ vanishing on $G_n$
(see \cite[Example~3.4]{Sal19}). He also provided
additional examples of hyperrigid sets annihilated by
states in the context of Cuntz-Krieger algebras (see
\cite[Section~5]{Sal19}). This line of research led
Salomon to introduce the concept of rigidity at $0$.
According to \cite[Definition~3.1]{Sal19}, a set $G$
generating a $C^*$-algebra $\ascr$ is said to be {\em
rigid at zero} in $\ascr$ if, for every sequence
$\{\psi_n\}_{n=1}^{\infty}$ of contractive positive
maps $\psi_n\colon \ascr \to \cbb$,
   \begin{align*}
\lim_{n\to\infty} \psi_n(g)=0 \;\; \forall g\in G \;\;
\implies \;\; \lim_{n\to\infty} \psi_n(a)=0 \;\;
\forall a\in \ascr.
   \end{align*}
If $\ascr$ is separable, then $G$ is rigid at $0$ if
and only if the states on $\ascr$ do not annihilate
$G$ (\cite[Theorem~3.3]{Sal19}). In a recent paper
\cite{De-Ka-Pa25}, such a set $G$ is called a
separating set. The distinction between these notions
consists in the fact that in \cite{De-Ka-Pa25} $G$ is
not required to contain the unit of $\ascr$ when
$\ascr$ is unital. One of our main results,
Theorem~\ref{soorh}, establishes a dichotomy that
holds for any hyperrigid set $G$ in a unital
$C^*$-algebra $\ascr$. It asserts that whenever
$\pi\colon \ascr \to \ogr(\hh)$ and $\rho\colon \ascr
\to \ogr(\kk)$ are representations, and $R \in
\ogr(\hh,\kk)$ is a non-isometric contraction such
that $\pi(g) = R^*\rho(g)R$ for all $g \in G$, then
the intertwining relation $R\pi=\rho R$ (i.e.,
$R\pi(a)=\rho(a)R$ for all $a\in \ascr$) holds if $G$
is annihilated by a state on $\ascr$, and it fails if
$G$ is not annihilated by any state on $\ascr$.
Exactly the same dichotomy is valid when the
intertwining relation $R\pi=\rho R$ is replaced by the
condition that $J=\big\{a\in\ascr :
\pi(a)=R^*\rho(a)R\big\}$ is an ideal. The first
alternative of the dichotomy is denoted $(\alpha)$,
and the second $(\beta)$. Alternative $(\beta)$ has
been studied in depth in \cite[Theorem~3.3]{Sal19} and
further developed in \cite{De-Ka-Pa25} (see, e.g.,
Theorem~3.19 therein). The present paper focuses on
alternative $(\alpha)$ of this dichotomy, allowing for
the more general situation in which the hyperrigid set
$G$ is annihilated by a nonzero completely positive
map (see Proposition~\ref{sweq}).

In a recent paper \cite{P-S-24}, we investigated, in
the context of a unital commutative $C^*$-algebra
$\ascr$ generated by a single element $t \in \ascr$,
which sets of the form
   \begin{align*}
G =\big\{t^{*m} t^n\colon (m,n) \in
\varXi\big\}, \quad \varXi \subseteq
\zbb_+^2,
   \end{align*}
are hyperrigid in $\ascr$, where, $\zbb_+$ denotes the
set of all nonnegative integers. We showed that, under
mild constraints on $\varXi$, the assumption that $G$
generates the $C^*$-algebra $\ascr$ implies that $G$
is hyperrigid (see \cite[Theorem~2.3]{P-S-24}). As a
consequence, we obtained two criteria for the
hyperrigidity of $G$: one depending on the geometry of
the spectrum of the generating element $t$, and the
other independent of it (see
\cite[Theorem~2.4]{P-S-24}). In particular, this
implies that for any character $\chi$ of a singly
generated commutative unital $C^*$-algebra $\ascr$,
there exists a finite hyperrigid set $G \subseteq
\ascr$ with $G \subseteq \ker \chi$ (see
Proposition~\ref{chyr}).

Taken together, these results provide motivation for
investigating the hyperrigidity of sets $G$ from the
perspective of states that annihilate them.
   \section{Main results}
We begin by recalling the following dilation-based
characterizations of the hyperrigidity of subsets of a
unital $C^*$-algebra. The equivalence
(i)$\Leftrightarrow$(ii) follows from the unique
extension property (see \cite[Theorem~2.1]{Arv11}) and
the Stinespring dilation theorem (see
\cite[Theorem~1]{Sti55}), upon observing that the
Hilbert space $\kk$ in conditions (ii) and (iii) may
be taken to be separable (see the proof of
Theorem~\ref{ogylu}). In turn, the equivalence
(ii)$\Leftrightarrow$(iii) is an immediate consequence
of \cite[Lemma~3.2]{P-S-roots23}. Throughout the
paper, $\ogr(\hh,\kk)$ denotes the Banach space of all
bounded linear operators from a Hilbert space $\hh$
into a Hilbert space $\kk$.
   \begin{theorem}   \label{delt}
   Let $G$ be a finite or countably infinite set of
generators of a unital $C^*$-algebra $\ascr$. Then the
following conditions are equivalent{\em :}
   \begin{enumerate}
   \item[(i)] $G$ is hyperrigid,
   \item[(ii)]
for all separable Hilbert spaces $\hh$ and $\kk$, all
representations $\pi\colon \ascr \to \ogr(\hh)$ and
$\rho\colon \ascr \to \ogr(\kk)$, and every isometry
$V\in \ogr(\hh,\kk)$,
   \begin{align*}
\pi(g) = V^*\rho(g)V \;\: \forall g\in G \implies
\pi(a) = V^*\rho(a)V \;\: \forall a\in \ascr,
   \end{align*}
   \item[(iii)]
for all separable Hilbert spaces $\hh$
and $\kk$, all representations
$\pi\colon \ascr \to \ogr(\hh)$ and
$\rho\colon \ascr \to \ogr(\kk)$, and
every isometry $V\in \ogr(\hh,\kk)$,
   \begin{align} \label{wvzq}
\pi(g) = V^*\rho(g)V \;\: \forall g\in
G \implies V\pi (a) = \rho(a) V \;\:
\forall a\in \ascr.
   \end{align}
   \end{enumerate}
Moreover, if $G$ is hyperrigid, then $\pi(a) =
V^*\rho(a)V$ for all $a\in \ascr$.
   \end{theorem}
If, in the if-clause of implication
\eqref{wvzq}, we replace the isometry
$V\in \ogr(\hh,\kk)$ with an arbitrary
operator $R\in \ogr(\hh,\kk)$, i.e.,
   \begin{align} \label{forfna}
\pi(g) = R^*\rho(g)R, \quad g\in G,
   \end{align}
then $R$ must be an isometry whenever $G$ contains the
unit $e$ of $\ascr$. However, the absence of $e$ in
$G$ allows for the possibility of non-isometric
solutions $R$ to \eqref{forfna}. Under the assumption
that $G$ is hyperrigid, the existence of such a
solution is closely related to the rigidity of $G$ at
$0$ and to the existence of a character of (or more
generally a state on) $\ascr$ vanishing on $G$. This
connection is made explicit in Corollaries~\ref{tiply}
and~\ref{hpse}. It turns out there is a dichotomy:
existence versus nonexistence of a state annihilating
$G$, with the corresponding behavior of the
intertwining relation described in each case. Given a
triple $(\pi,\rho,R)$, where $\pi\colon \ascr \to
\ogr(\hh)$ and $\rho\colon \ascr \to \ogr(\kk)$ are
representations and $R \in \ogr(\hh,\kk)$ is a
contraction, we set\footnote{By automatic continuity
of representations of $C^*$-algebras, $J$ is closed.
Moreover, $R$ is an isometry if and only if
$J=\ascr$.}
   \begin{align} \label{zbyr}
J := \big\{a \in \ascr : \pi(a) = R^*\rho(a)R\big\}.
   \end{align}
For brevity, we suppress the dependence of $J$ on
$(\pi,\rho,R)$. Note that, in view of
Theorem~\ref{delt}, if $G$ is hyperrigid and $R$ is an
isometry, then $R\pi = \rho R$ and $J = \ascr$. For
later use, we associate with any contraction $R \in
\ogr(\hh,\kk)$ the operators
   \begin{align} \label{dwytry}
\triangle_{R} := I - R^*R, \qquad \triangle_{R^*} := I
- RR^*.
   \end{align}
The operators $\triangle_{R}^{1/2}$ and
$\triangle_{R^*}^{1/2}$ are called the defect
operators of $R$ (cf.~\cite{Sz-F-B-K10}).
   \begin{theorem} \label{soorh}
Let $G$ be a hyperrigid subset of a unital
$C^*$-algebra $\ascr$. Suppose that $\hh$ and $\kk$
are separable Hilbert spaces, $\pi\colon \ascr \to
\ogr(\hh)$ and $\rho\colon \ascr \to \ogr(\kk)$ are
representations, and $R \in \ogr(\hh,\kk)$ is a
non-isometric contraction such that $\pi(g) =
R^*\rho(g)R$ for all $g \in G$. Then the following
dichotomy~holds{\em :}
   \begin{enumerate}
   \item[($\alpha$)] Either there exists a state $\phi$ on $\ascr$ such that
$G \subseteq \ker \phi$, in which case $R \pi(a) =
\rho(a) R$ for all $a\in \ascr$ and $J$ is a closed
ideal in $\ascr$.
   \item[($\beta$)] Or, for every state $\phi$ on $\ascr$,
$G \nsubseteq \ker \phi$, in which case neither
$R\pi(a)=\rho(a)R$ for all $a \in \ascr$ nor $J$ is an
ideal.
   \end{enumerate}
Moreover, if {\em ($\alpha$)} holds, then $\ker \phi =
J$ and $\phi$ is a character of $\ascr$, and
   \begin{align} \label{xzaw}
\pi(a) = R^*\rho(a)R + \phi(a) \triangle_{R}, \quad a
\in \ascr.
   \end{align}
   \end{theorem}
Theorem~\ref{soorh} can be stated in the following
logically equivalent form.
   \begin{manualtheorem}
   Let $G$ be a hyperrigid subset of a unital
$C^*$-algebra $\ascr$, and let $(\pi,\rho,R)$ satisfy
the assumptions of Theorem~{\em \ref{soorh}}. Then the
following are equivalent{\em :}
   \begin{enumerate}
   \item[(i)] there exists a state on $\ascr$
vanishing on $G$,
   \item[(ii)] there exists a character of $\ascr$
vanishing on $G$,
   \item[(iii)] $R$ intertwines $\pi$ and $\rho$,
and/or $J$ is an ideal.
   \end{enumerate}
   \end{manualtheorem}
We refer the reader to Section~\ref{9.e} for examples
illustrating Theorem~\ref{soorh}($\beta$).

It turns out that even in the unital case, when the
generating set $G$ does not contain the unit,
completely contractive completely positive (CCCP) maps
arise naturally in the study of hyperrigidity (see
\cite{Sal19,De-Ka-Pa25}). This is established in
Theorem~\ref{ogylu} below. Before proceeding, recall
that by \cite[Theorem~2.1.7]{Mur90} and the
Stinespring dilation theorem \cite[Theorem~1]{Sti55}:
   \begin{align} \label{stymyr}
   \begin{minipage}{53ex}
   \begin{center}
{\em A linear map $\varPsi\colon \ascr\to \ogr(\hh)$
is CCCP if and only if $\varPsi$ is completely
positive and $\|\varPsi(e)\|\Le 1$.}
   \end{center}
   \end{minipage}
   \end{align}
Given a unital $C^*$-algebra $\ascr$, we denote by
$\mathfrak{M}_{\ascr}$ the set of all characters of
$\ascr$, i.e., the set of all one-dimensional
representations $\chi\colon \ascr \to \cbb$. The
assumption that $G \subseteq \ker\phi$ plays an
essential role in the theorem below (see
Remark~\ref{qqre}).
   \begin{theorem} \label{ogylu}
   Let $G$ be a finite or countably
infinite set of generators of a unital
$C^*$-algebra $\ascr$, and let $\phi$
be a state on $\ascr$ such that $G
\subseteq \ker\phi$. Then the following
conditions are equivalent{\em :}
   \begin{enumerate}
   \item[(i)] $G$ is hyperrigid,
   \item[(ii)] for all separable Hilbert spaces $\hh$ and $\kk$,
all representations $\pi\colon \ascr \to \ogr(\hh)$
and $\rho\colon \ascr \to \ogr(\kk)$ and every
contraction $R\in \ogr(\hh,\kk)$,
   \begin{align} \label{nicuntr}
\pi(g) = R^*\rho(g)R \;\; \forall g\in G \implies
R\pi(a)=\rho(a)R \;\; \forall a\in \ascr,
   \end{align}
   \item[(iii)] for all separable Hilbert spaces $\hh$ and $\kk$,
all representations $\pi\colon \ascr \to \ogr(\hh)$
and $\rho\colon \ascr \to \ogr(\kk)$ and every
contraction $R\in \ogr(\hh,\kk)$,
   \begin{align} \label{ramc}
\pi(g) = R^*\rho(g)R \;\; \forall g\in G \implies
\pi(a)=R^*\rho(a)R \;\; \forall a \in \mcal,
   \end{align}
where $\mcal = \ascr$ if $R^*R=I${\em ;} otherwise,
$\mcal = \ker\chi$ for some $\chi\in
\mathfrak{M}_{\ascr}$.
   \item[(iv)] for every separable Hilbert space $\hh$,
every representation $\pi\colon \ascr \to \ogr(\hh)$
and every CCCP map $\varPsi\colon \ascr \to
\ogr(\hh)$,
   \begin{align*}
\pi(g) = \varPsi(g) \;\; \forall g\in G \implies
\pi(a)=\varPsi(a) \;\; \forall a \in \mcal,
   \end{align*}
where $\mcal = \ascr$ if $\varPsi(e)=I$; otherwise,
$\mcal = \ker\chi$ for some $\chi\in
\mathfrak{M}_{\ascr}$.
   \end{enumerate}
Moreover, if $G$ is hyperrigid and the operator $R$ in
{\em (iii)} is not isometric, or if $\varPsi(e)\neq I$
in {\em (iv)}, then $\phi=\chi$.
   \end{theorem}
The corollary below shows that the characterization in
\cite[Theorem~3.19(vii)]{De-Ka-Pa25}, stated via the
CCCP unique extension property for separating sets $G$
(i.e., those satisfying alternative $(\beta)$ of our
dichotomy), remains valid when $G$ is as in
alternative $(\alpha)$, provided that the
$C^*$-algebra in question is character-free. This is
the case in particular for simple unital
$C^*$-algebras (see \cite{Mur90}). For examples of
noncommutative unital $C^*$-algebras that either admit
characters or do not, we refer the reader to
\cite[Theorem~7.23]{Doug72} for the former case and to
\cite{Cun77} and \cite[Corollary~I.10.6]{Dav96} for
the latter. It is worth noting that the Cuntz
$C^*$-algebra $\mathcal O_n$ with $n \in \natu
\setminus \{1\}$ falls within the scope of
Corollary~\ref{caly}, since $\mathcal O_n$ is simple
(see \cite[Theorem~1.13]{Cun77}), its standard
generating set $G_n = \{V_1, \ldots, V_n\}$ is
hyperrigid (see \cite[Corollary~3.4]{Arv11}), and
$G_n$ is annihilated by a state on $\mathcal O_n$ (see
\cite[Example~3.4]{Sal19}).
   \begin{corollary} \label{caly}
   Let $\ascr$ be a unital $C^*$-algebra with no
   characters. Suppose that $G$ is a finite or
   countably infinite generating set of $\ascr$ such
   that $G \subseteq \ker\phi$ for some state $\phi$
   on $\ascr$. Then the following conditions are
   equivalent{\em :}
   \begin{enumerate}
   \item[(i)] $G$ is hyperrigid,
   \item[(ii)] for all separable Hilbert spaces $\hh$ and $\kk$, all
representations $\pi\colon \ascr \to \ogr(\hh)$ and
$\rho\colon \ascr \to \ogr(\kk)$ and every contraction
$R\in \ogr(\hh,\kk)$,
   \begin{align*}
\pi(g) = R^*\rho(g)R \;\; \forall g\in G \implies
\pi(a)=R^*\rho(a)R \;\; \forall a \in \ascr,
   \end{align*}
   \item[(iii)] for every separable Hilbert space $\hh$, every
representation $\pi\colon \ascr \to \ogr(\hh)$ and
every CCCP map $\varPsi\colon \ascr \to \ogr(\hh)$,
   \begin{align*}
\pi(g) = \varPsi(g) \;\; \forall g\in G \implies
\pi(a) = \varPsi(a) \;\; \forall a\in \ascr.
   \end{align*}
   \end{enumerate}
Moreover, if $G$ is hyperrigid, then $R$ in {\em (ii)}
is isometric and $\varPsi$ in {\em (iii)} is a
UCP~map.
   \end{corollary}
Note that the ``moreover'' part of
Corollary~\ref{caly} may still hold even if the
$C^*$-algebra in question has many characters,
provided the remaining assumptions are unchanged.
Indeed, consider the unital $C^*$-algebra $C(\tbb)$,
which is generated by the identity function $\xi$ on
$\tbb = \{z\in\cbb\colon |z|=1\}$. In contrast to the
Cuntz algebra $\mathcal O_n$ with $n \in
\natu\setminus \{1\}$, the algebra $C(\tbb)$ has
infinitely many characters; nevertheless, $G=\{\xi\}$
is hyperrigid in $C(\tbb)$ (see \cite[Corollary~2,
p.~795]{Sh20}) and is annihilated by the state $\phi$
defined by $\phi(f)=\int_{\tbb} f \D\mu$ for $f \in
C(\tbb)$, where $\mu$ is the normalized Lebesgue
measure on $\tbb$ (cf.~\cite[Example~1.3]{Sal19}).
Since the characters of $C(\tbb)$ never annihilate
$\xi$, it follows from Theorem~\ref{soorh} that the
``moreover'' part of Corollary~\ref{caly} still holds
in this case.

   To state the next theorem, we need the definition
of a $\chi$-decomposition.
   \begin{df} \label{dsfim}
Let $\ascr$ be a unital $C^*$-algebra,
$\chi$ be a character of $\ascr$,
$\pi\colon \ascr \to \ogr(\hh)$ and
$\rho\colon \ascr \to
\ogr(\mathcal{K})$ be representations
on Hilbert spaces $\hh$ and $\kk$,
respectively, and $R \in
\ogr(\mathcal{H}, \mathcal{K})$ be a
contraction. Suppose that $\hh_0$ and
$\hh_1$ are closed subspaces of $\hh$,
and $\kk_0$ and $\kk_1$ are closed
subspaces of $\kk$. We say that $(\pi,
\rho, R)$ has an {\em isometric}
(resp., a {\em unitary}) {\em
$\chi$-decomposition} $[$relative to
$(\hh_0, \hh_1, \kk_0, \kk_1$)$]$ if
   \begin{enumerate}
   \item[(i)] $\hh = \hh_0 \oplus \hh_1$ and $\kk = \kk_0 \oplus
\kk_1$,
   \item[(ii)] $\hh_0$ and $\hh_1$ reduce $\pi$, and
$\kk_0$ and $\kk_1$ reduce $\rho$,
   \item[(iii)] $\pi(a)|_{\hh_1}=\chi(a) I_{\hh_1}$ and
$\rho(a)|_{\kk_1}=\chi(a) I_{\kk_1}$ for all $a\in
\ascr$,
   \item[(iv)] $R=R_0 \oplus R_1$, where $R_0 \in
\ogr(\hh_0,\kk_0)$ is isometric $($resp., unitary$)$
and $R_1 \in \ogr(\hh_1,\kk_1)$,
   \item[(v)] $\pi(a)|_{\hh_0}  = R_0^*\rho(a)|_{\kk_0}
R_0$ for all $a \in \ascr$.
   \end{enumerate}
   \end{df}
The cases $\hh_0=\{0\}$ and $\hh_1=\{0\}$ are not
excluded.
   \begin{theorem} \label{rconrep} 
Let $G$ be a finite or countably
infinite set of generators of a unital
$C^*$-algebra $\ascr$, and let $\chi$
be a character of $\ascr$ such that $G
\subseteq \ker\chi$. Then the following
conditions are equivalent{\em :}
   \begin{enumerate}
   \item[(i)] $G$ is hyperrigid,
   \item[(ii)] for all separable Hilbert spaces $\hh$ and $\kk$,
all representations $\pi\colon \ascr \to
\ogr(\mathcal{H})$ and $\rho\colon \ascr \to
\ogr(\mathcal{K})$ and every contraction $R
\in \ogr(\mathcal{H}, \mathcal{K})$,
   \begin{align} \label{plyc}
\pi(g) = R^* \rho(g) R \;\; \forall g \in G
\implies \pi(a) = R^* \rho(a) R + \chi(a)
\triangle_R \;\; \forall a \in \ascr,
   \end{align}
   \item[(iii)] for all separable Hilbert spaces $\hh$ and $\kk$,
all representations $\pi\colon \ascr \to \ogr(\hh)$
and $\rho\colon \ascr \to \ogr(\kk)$ and every
contraction $R\in \ogr(\hh,\kk)$,
   \begin{align*}
\pi(g) = R^*\rho(g)R \;\; \forall g\in G \implies
\pi(a)=R^*\rho(a)R \;\; \forall a\in \ker\chi,
   \end{align*}
   \item[(iv)] for all separable Hilbert spaces $\hh$ and $\kk$,
all representations $\pi\colon \ascr \to
\ogr(\mathcal{H})$ and $\rho\colon \ascr \to
\ogr(\mathcal{K})$ and every contraction $R \in
\ogr(\mathcal{H}, \mathcal{K})$,
   \begin{align*}
\pi(g) = R^* \rho(g) R \;\; \forall g \in G \implies
(\pi, \rho, R) \text{ has an isometric }
\chi\text{-decomposition,}
   \end{align*}
   \item[(v)] for all separable Hilbert spaces $\hh$ and $\kk$,
all representations $\pi\colon \ascr \to
\ogr(\mathcal{H})$ and $\rho\colon \ascr \to
\ogr(\mathcal{K})$ and every contraction $R \in
\ogr(\mathcal{H}, \mathcal{K})$ satisfying the
minimality condition $\kk=\bigvee_{a\in \ascr}
\rho(a)R(\hh)$,
   \begin{align*}
\pi(g) = R^* \rho(g) R \;\; \forall g \in G \implies
(\pi, \rho, R) \text{ has a unitary }
\chi\text{-decomposition.}
   \end{align*}
   \end{enumerate}
   \end{theorem}

As a complement to the first part ($\alpha$) of the
dichotomy in Theorem~\ref{soorh}, we obtain the
following $\chi$-decomposition. Recall that an
operator $R \in \ogr(\hh)$ is called a \emph{pure
contraction} if $\|Rh\| < \|h\|$ for every $h \in \hh
\setminus \{0\}$ (see \cite{Ka83}).
   \begin{theorem} \label{soorh-2}
Let $G$ be a hyperrigid subset of a unital
$C^*$-algebra $\ascr$, and let $\chi$ be a character
of $\ascr$ such that $G \subseteq \ker\chi$. Suppose
that $\hh$ and $\kk$ are separable Hilbert spaces,
$\pi\colon \ascr \to \ogr(\hh)$ and $\rho\colon \ascr
\to \ogr(\kk)$ are representations, and $R \in
\ogr(\hh,\kk)$ is a non-isometric contraction such
that $\pi(g) = R^*\rho(g)R$ for all $g \in G$. Then
   \begin{enumerate}
   \item[(i)] $\triangle_{R}$ commutes with $\pi$
and $\triangle_{R^*}$ commutes with $\rho$ $($see
\eqref{dwytry}$)$,
   \item[(ii)] $(\pi,\rho,R)$ has an isometric
$\chi$-decomposition relative to $(\hh_0, \hh_1,
\kk_0, \kk_1)$, where $\hh_0 = \hh_1^{\perp}$, $\hh_1
= \overline{\ran(\triangle_{R})}$, $\kk_0 =
\kk_1^{\perp}$, and $\kk_1 =
\overline{\ran(\triangle_{R^*}R)}$,
   \item[(iii)] $R_1 \in \ogr(\hh_1, \kk_1)$ is a
pure contraction with dense range.
   \end{enumerate}
   \end{theorem}
See Proposition~\ref{ghkds} for a converse (in a
certain sense) to Theorem~\ref{soorh-2}. Other
$\chi$-decompositions are established in
Theorem~\ref{rroiu} and Corollary~\ref{uspc} below.
Before stating them, we define
   \begin{align} \label{mycj}
\mcal_{\chi}(\pi) = \bigcap_{a\in \ascr} \nul(\pi(a) -
\chi(a) I),
   \end{align}
where $\pi\colon \ascr \to \ogr(\hh)$ is a
representation of a unital $C^*$-algebra $\ascr$ on a
Hilbert space $\hh$, and $\chi$ is a character of
$\ascr$. The space $\mcal_{\chi}(\pi)$ may be viewed
as a generalization of a joint eigenspace of a
$d$-tuple of commuting operators (see
Lemma~\ref{versyn}).
   \begin{theorem} \label{rroiu}
Let $G$ be a hyperrigid subset of a unital
$C^*$-algebra $\ascr$, and let $\chi$ be a character
of $\ascr$ such that $G \subseteq \ker\chi$. Suppose
that $\pi\colon \ascr \to \ogr(\hh)$ and $\rho\colon
\ascr \to \ogr(\kk)$ are representations of $\ascr$ on
separable Hilbert spaces $\hh$ and $\kk$,
respectively, and $R \in \ogr(\mathcal{H},\kk)$ is a
contraction such~that $\pi(g) = R^* \rho(g) R$ for
every $g\in G$. Then the following statements hold{\em
:}
   \begin{itemize}
    \item[(i)] $\mcal_{\chi}(\pi)$ reduces $\pi$, $R^*R$ and
$\triangle_{R}$, where $\mcal_{\chi}(\pi)$
is as in \eqref{mycj},
   \item[(ii)]  $\lcal_{\chi} :=
\bigvee_{a \in \ascr} {\rho(a)} R
(\mcal_{\chi}(\pi))$ and
$\lcal_{\chi_{\perp}} := \bigvee_{a \in
\ascr} {\rho(a)}R (\mcal_{\chi}(\pi)^\perp)$
reduce~$\rho$,
   \item[(iii)] $\lcal_{\chi} =
\overline{R(\mcal_{\chi}(\pi))}$,
$\lcal_{\chi_{\perp}} =
\overline{R(\mcal_{\chi}(\pi)^{\perp})}$ and
$\lcal_{\chi} \perp \lcal_{\chi_{\perp}}$,
    \item[(iv)] $R(\mcal_{\chi}(\pi)) \subseteq \mcal_{\chi}(\rho)$ and $\ran(\triangle_{R})
\subseteq \mcal_{\chi}(\pi)$,
    \item[(v)] $(\pi,\rho,R)$ has an isometric
$\chi$-decomposition relative to $(\hh_0,
\hh_1, \kk_0, \kk_1)$, where
$\hh_0=\mcal_{\chi}(\pi)^\perp$,
$\hh_1=\mcal_{\chi}(\pi)$,
$\kk_0=\lcal_{\chi}^\perp$ and
$\kk_1=\lcal_{\chi}$,
    \item[(vi)] $R_0
\pi(a)|_{\hh_0} = \rho(a)|_{\kk_0} R_0$
for all $a \in \ascr.$
   \end{itemize}
Moreover, if $\kk = \bigvee_{a \in \ascr}
\rho(a)R(\hh)$, then $\lcal_{\chi_{\perp}} =
\lcal_{\chi}^\perp$ and $R_0$ is unitary, so
$(\pi,\rho,R)$ has a unitary
$\chi$-decomposition relative to $(\hh_0,
\hh_1, \kk_0, \kk_1)$.
   \end{theorem}
Finally, we combine the two orthogonal
decompositions given in
Theorems~\ref{rconrep} and \ref{rroiu}.
As a consequence, we obtain a
decomposition into three parts. Below,
$\dfi(V)$ denotes the {\em defect
space} of an isometry $V\in
\ogr(\hh,\kk)$, defined by
   \begin{align*}
\dfi(V)=\ran(V)^{\perp} = \nul(V^*).
   \end{align*}
   \begin{corollary} \label{uspc}
Let $G$ be a hyperrigid subset of a unital
$C^*$-algebra $\ascr$, and let $\chi$ be a character
of $\ascr$ such that $G \subseteq \ker\chi$. Suppose
that $\pi\colon \ascr \to \ogr(\hh)$ and $\rho\colon
\ascr \to \ogr(\kk)$ are representations of $\ascr$ on
separable Hilbert spaces $\hh$ and $\kk$,
respectively, and $R \in \ogr(\mathcal{H},\kk)$ is a
contraction such that $\pi(g) = R^* \rho(g) R$ for
every $g\in G$, and $\kk = \bigvee_{a \in \ascr}
\rho(a)R(\hh)$. Then
   \begin{enumerate}
   \item[(i)]  $\hh = \hh_I \oplus \hh_{PC}$ and
$\hh_I = \hh_U \oplus \hh_S$, where
$\hh_{I}=\nul(\triangle_R)$, $\hh_{U} =
\mcal_{\chi}(\pi)^{\perp}$, $\hh_{S} =
\nul(\triangle_R) \cap \mcal_{\chi}(\pi)$,
$\hh_{PC}= \overline{\ran(\triangle_R)}$,
   \item[(ii)] $\kk = \kk_{I} \oplus \kk_{PC}$ and
$\kk_{I} = \kk_{U} \oplus \kk_{S}$, where
$\kk_{I} = \overline{\ran(R
\triangle_R)}^{\perp}$,
$\kk_{U}=\lcal_{\chi}^{\perp}$, $\kk_{S} =
\overline{\ran(R\triangle_R)}^{\perp} \cap
\lcal_{\chi}$, $\kk_{PC} =
\overline{\ran(R\triangle_R)}$
$($$\lcal_{\chi}$ is as in {\em
Theorem~\ref{rroiu}}$)$,
   \item[(iii)] the spaces $\hh_{U}$,
$\hh_{S}$ and $\hh_{PC}$ reduce $\pi$, and
$\pi|_{\hh_{U}^{\perp}}=\chi
I_{\hh_{U}^{\perp}}$,
   \item[(iv)] the spaces $\kk_{U}$,
$\kk_{S}$ and $\kk_{PC}$ reduce $\rho$, and
$\rho|_{\kk_{U}^{\perp}}=\chi
I_{\kk_{U}^{\perp}}$,
   \item[(v)] $R=R_{I} \oplus R_{PC}$ and
$R_{I}=R_{U} \oplus R_{S}$, where $R_{I} \in
\ogr(\hh_{I},\kk_{I})$, $R_{U} \in
\ogr(\hh_{U},\kk_{U})$, $R_{S}\in
\ogr(\hh_{S},\kk_{S})$ and $R_{PC}\in
\ogr(\hh_{PC}, \kk_{PC})$,
   \item[(vi)] $R_{U}$ is unitary, $R_{S}$ is an isometry
and $R_{PC}$ is a pure contraction,
   \item[(vii)] $\pi(a)|_{\hh_{U}} = R_{U}^* \rho(a)|_{\kk_{U}}
R_{U}$ for every $a \in \ascr$,
   \item[(viii)] $\dfi(R_{I}) = \dfi(R_{S})=
\nul(R^*)$.
   \end{enumerate}
   \end{corollary}
   \begin{rem} In all results of this paper where
the set $G$ appears (in particular in this section,
the next section, and Section \ref{Sec.5.5}), $G$ is
assumed to be a countable set of generators of the
underlying $C^*$-algebra. Hence the $C^*$-algebra is
automatically separable. A careful inspection of the
proofs shows that these results remain valid even when
the Hilbert spaces $\hh$ and $\kk$ are not assumed to
be separable (cf. \cite[Appendix B]{P-S-24}).
   \hfill $\diamondsuit$
   \end{rem}
The proofs of Theorems~\ref{soorh} and~\ref{soorh-2}
are given in Section~\ref{Sec.5}. In this section we
also show that translating elements of a hyperrigid
set by complex numbers typically leads to a situation
in which the intertwining relation $R\pi=\rho R$ fails
to hold and the set $J$ is not an ideal (see
Proposition~\ref{glmp}). The proofs of
Theorem~\ref{ogylu} and Corollary~\ref{caly} are
presented in Section~\ref{Sec.5.5}. This section also
provides characterizations of hyperrigidity by means
of generalized joint eigenspaces associated with the
underlying representations (see
Propositions~\ref{gfza} and~\ref{jger}). The proofs of
Theorems~\ref{rconrep} and~\ref{rroiu} and
Corollary~\ref{uspc} are in Section~\ref{Sect.8},
where two corollaries to Theorem~\ref{rconrep} are
also included.
   \section{Hyperrigidity via $d$-tuples of operators}
   A $d$-tuple $\tbold=(T_j)_{j=1}^d \in \ogr(\hh)^d$
of operators on a Hilbert space $\hh$ is called {\em
normal} if each component $T_j$ is normal and the
components commute with one another. Let $C(X)$ denote
the $C^*$-algebra of all continuous complex-valued
functions on $X$, equipped with the supremum norm.
   \begin{prop} \label{dtdw}
Let $X$ be a compact subset of $\cbb^d$ $($with $d\in
\natu$$)$, and let $\lambdab \in X$. Suppose that $G$
is a finite or countably infinite set of generators of
$C(X)$ such that $f(\lambdab)=0$ for every $f\in G$.
Then the following conditions are equivalent{\em :}
   \begin{itemize}
   \item[(i)] $G$ is hyperrigid in $C(X)$,
   \item[(ii)] for all separable Hilbert spaces $\hh$ and $\kk$,
all normal $d$-tuples
$\tbold=(T_j)_{j=1}^d \in \ogr(\hh)^d$
and $\nbold =(N_j)_{j=1}^d\in
\ogr(\kk)^d$ with Taylor spectrum in
$X$, and for every contraction $R\in
\ogr(\hh,\kk)$, if either
$\lambdab\notin
\sigma_{\mathsf{p}}(\tbold)$ or
$T_j=\lambda_j I$ for all $j=1,\ldots,
d$, then
   \begin{align*}
f(\tbold)=R^*f(\nbold)R \;\; \forall f\in G \implies
\;\; RT_j=N_jR \;\; \forall j=1,\ldots, d.
   \end{align*}
   \end{itemize}
   \end{prop}
   \begin{prop} \label{dttr}
Let $X$ be a compact subset of $\cbb^d$ $($with $d\in
\natu$$)$, and let $\lambdab \in X$. Suppose that $G$
is a finite or countably infinite set of generators of
$C(X)$ such that $f(\lambdab)=0$ for every $f\in G$,
and that there exists $g_0\in G$ such that
$g_0^{-1}((0,\infty)) = X \setminus \{\lambdab\}$.
Then the following conditions are equivalent{\em :}
   \begin{itemize}
   \item[(i)] $G$ is hyperrigid in $C(X)$,
   \item[(ii)] for all separable Hilbert spaces $\hh$ and $\kk$,
all normal $d$-tuples
$\tbold=(T_j)_{j=1}^d \in \ogr(\hh)^d$
and $\nbold =(N_j)_{j=1}^d \in
\ogr(\kk)^d$ with Taylor spectrum in
$X$, and for every contraction $R\in
\ogr(\hh,\kk)$, if $\lambdab\notin
\sigma_{\mathsf{p}}(\tbold)$, then
   \begin{align*}
f(\tbold)=R^*f(\nbold)R \;\; \forall f\in G \implies
\;\; RT_j=N_jR \;\; \forall j=1,\ldots, d.
   \end{align*}
   \end{itemize}
   \end{prop}
Before stating the next result, we
introduce the following definition.
   \begin{df} 
Let $X$ be a nonempty compact subset of
$\cbb^d$ (with $d \in \natu$), $\hh$
and $\kk$ be Hilbert spaces, $\tbold =
(T_j)_{j=1}^d \in \ogr(\hh)^d$ and
$\nbold = (N_j)_{j=1}^d \in
\ogr(\kk)^d$ be normal $d$-tuples with
Taylor spectrum in $X$, $R\in
\ogr(\hh,\kk)$ be a contraction and
$\lambdab = (\lambda_1, \ldots,
\lambda_d) \in X$. Suppose that $\hh_0$
and $\hh_1$ are closed subspaces of
$\hh$, and $\kk_0$ and $\kk_1$ are
closed subspaces of $\kk$. We say that
$(\tbold, \nbold, R)$ has an isometric
(resp., a unitary)
$\lambdab$-decomposition $[$relative to
$(\hh_0, \hh_1, \kk_0, \kk_1$)$]$ if
   \begin{enumerate}
   \item[(i)] $\hh = \hh_0 \oplus \hh_1$ and $\kk = \kk_0 \oplus
\kk_1$,
   \item[(ii)] $\hh_0$ and $\hh_1$ reduce each $T_j$, and
$\kk_0$ and $\kk_1$ reduce each $N_j$ for $j=1,
\ldots, d$,
   \item[(iii)] $T_j|_{\hh_1} = \lambda_j I_{\hh_1}$ and
$N_j|_{\kk_1} = \lambda_j I_{\kk_1}$ for $j=1, \ldots,
d$,
   \item[(iv)] $R = V \oplus R_1$, where $V \in
\ogr(\hh_0,\kk_0)$ is isometric
$($resp., unitary$)$ and $R_1 \in
\ogr(\hh_1,\kk_1)$,
   \item[(v)] $T_j|_{\hh_0} = V^*N_j|_{\kk_0}
V$ for $j=1, \ldots, d$.
   \end{enumerate}
   \end{df}
The following theorem is the operator $d$-tuple
counterpart of the equivalences
(i)$\Leftrightarrow$(ii)$\Leftrightarrow$(v) in
Theorem~\ref{rconrep}.
   \begin{prop} \label{kridut}
Let $X$ be a compact subset of $\cbb^d$ $($with $d\in
\natu$$)$, and let $\lambdab \in X$. Suppose that $G$
is a finite or countably infinite set of generators of
$C(X)$ such that $f(\lambdab)=0$ for every $f\in G$.
Then the following conditions are equivalent{\em :}
   \begin{itemize}
   \item[(i)] $G$ is hyperrigid in $C(X)$,
   \item[(ii)] for all separable Hilbert spaces $\hh$ and $\kk$,
all normal $d$-tuples $\tbold\in
\ogr(\hh)^d$ and $\nbold\in
\ogr(\kk)^d$ with Taylor spectrum in
$X$, and every contraction $R\in
\ogr(\hh,\kk)$,
   \begin{align*}
f(\tbold)=R^*f(\nbold)R \;\; \forall f\in G \implies
\;\; f(\tbold)=R^* f(\nbold) R + f(\lambdab)
\triangle_{R} \;\; \forall f\in C(X),
   \end{align*}
   \item[(iii)] for all separable Hilbert spaces $\hh$ and $\kk$,
all normal $d$-tuples $\tbold \in
\ogr(\hh)^d$ and $\nbold \in
\ogr(\kk)^d$ with Taylor spectrum in
$X$, and every contraction $R\in
\ogr(\hh,\kk)$ such that $\kk =
\bigvee_{\substack{\boldsymbol{\alpha},
\boldsymbol{\beta} \in \zbb_+^d}}
\nbold^{*\boldsymbol{\alpha}}
\nbold^{\boldsymbol{\beta}} R(\hh)$,
   \begin{align*}
f(\tbold) = R^* f(\nbold) R \;\;
\forall f\in G \implies (\tbold,
\nbold, R) \;\; \text{has a unitary
$\lambdab$-decomposition}.
   \end{align*}
   \end{itemize}
   \end{prop}
   \begin{prop} \label{dwct}
Let $X$ be a nonempty compact subset of $\cbb^d$
$($$d\in \natu$$)$, let $\mu$ be a Borel probability
measure on $X$, and let $G$ be a finite or countably
infinite set of generators of $C(X)$ such that
$\int_{X} f \D \mu=0$ for every $f\in G$.
Then the following conditions are
equivalent{\em :}
   \begin{enumerate}
   \item[(i)] $G$ is hyperrigid,
   \item[(ii)] for all separable Hilbert spaces $\hh$ and $\kk$,
all normal $d$-tuples $\tbold =
(T_j)_{j=1}^d \in \ogr(\hh)^d$ and
$\nbold = (N_j)_{j=1}^d \in
\ogr(\kk)^d$ with Taylor spectrum in
$X$, and every contraction $R\in
\ogr(\hh,\kk)$,
   \begin{align} \label{nykixx}
f(\tbold) = R^*f(\nbold) R \;\; \forall f\in
G \implies RT_j=N_jR \text{ for } j=1,
\ldots, d.
   \end{align}
   \end{enumerate}
Moreover, if {\em (i)} holds and $\tbold, \nbold, R$
satisfy the if-clause of \eqref{nykixx} with
non-isometric contraction $R$, then $\mu$ is the Dirac
measure at uniquely determined point of $X$.
   \end{prop}
Our next result is inspired by the characterizations
of operator convex functions given by Hansen and
Pedersen \cite{Ha-Pe82} (see also
\cite[Theorem~1.5]{P-S-roots23}).
   \begin{prop} \label{axer}
Let $X$ be a nonempty compact subset of $\cbb$, let
$\mu$ be a Borel probability measure on $X$, and let
$G$ be a finite or countably infinite set of
generators of $C(X)$ such that $\int_{X} f \D \mu=0$
for every $f\in G$. Then, for every integer $n\Ge 2$,
the following conditions are equivalent{\em :}
   \begin{enumerate}
   \item[(i)] $G$ is hyperrigid,
   \item[(ii)] for all separable Hilbert spaces $\hh$ and $\kk$,
all normal operators $T\in \ogr(\hh)$ and $N\in
\ogr(\kk)$ with spectra in $X$, and every contraction
$R\in \ogr(\hh,\kk)$,
   \begin{align} \label{nykixe}
f(T) = R^*f(N)R \;\; \forall f\in G \implies RT=NR,
   \end{align}
   \item[(iii)]
for all separable Hilbert spaces $\hh$, $\kk_1,
\ldots, \kk_n$, all normal operators $T\in \ogr(\hh)$,
$N_1\in \ogr(\kk_1), \ldots, N_n\in \ogr(\kk_n)$ with
spectra in $X$, and all operators $R_1 \in \ogr(\hh,
\kk_1), \dots, R_n \in \ogr(\hh, \kk_n)$ such that
$\sum_{i=1}^n R_i^*R_i \Le I$,
   \begin{align} \label{ftaini}
f(T) = \sum_{i=1}^{n} R_i^*f(N_i)R_i \;\; \forall f\in
G \implies R_iT=N_iR_i \;\; \forall i\in \{1, \ldots,
n\}.
   \end{align}
   \end{enumerate}
Moreover, if $G$ is hyperrigid and $T,N,R$ $($resp.,
$T, N_1, \ldots, N_n$, $R_1, \ldots, R_n$\/$)$ satisfy
the if-clause of \eqref{nykixe} $($resp.,
\eqref{ftaini}$)$ with $R^*R \lneqq I$ $($resp.,
$\sum_{i=1}^n R_i^*R_i \lneqq I$\/$)$, then $\mu$ is
the Dirac measure at a uniquely determined point of
$X$.
   \end{prop}
The proofs of Propositions~\ref{dtdw}, \ref{dttr},
\ref{kridut}, \ref{dwct} and~\ref{axer} appear in
Section~\ref{Sec.7}.
   \section{
Prerequisites}
   The field of complex numbers is denoted by $\cbb$.
By an ideal in an algebra we always mean a two-sided
ideal. It is well-known that every closed ideal in a
$C^*$-algebra is selfadjoint (see
\cite[Theorem~3.1.3]{Mur90}). Representations of
unital $C^*$-algebras are assumed to preserve
involutions and units. A completely positive (linear)
map between unital $C^*$-algebras that preserves units
is called a {\em unital completely positive $($UCP$)$}
map. We will abbreviate ``completely contractive
completely positive map'' between unital
$C^*$-algebras to ``CCCP map''.

Let $\hh$ and $\kk$ be (complex)
Hilbert spaces. Denote by $\ogr(\hh,
\kk)$ the Banach space of all bounded
linear operators from $\hh$ to $\kk$.
If $A\in \ogr(\hh,\kk)$, then $A^*$,
$\nul(A)$ and $\ran(A)$ stand for the
adjoint, the kernel and the range of
$A$, respectively. We say that $A\in
\ogr(\hh)$ is \textit{normal} if
$A^*A=AA^*$, \textit{selfadjoint} if
$A=A^*$ and \textit{positive} if
$\langle Ah,h\rangle \Ge 0$ for all
$h\in \hh$. Each positive operator
$A\in \ogr(\hh)$ has a unique positive
square root denoted by $A^{1/2}$. If
$A\in \ogr(\hh)$, we write
$|A|:=(A^*A)^{1/2}$.

Given a $d$-tuple $\tbold=(T_j)_{j=1}^d
\in \ogr(\hh)^d$ of commuting operators
(with $d\in \natu$), we write
$\tbold^*=(T_j^*)_{j=1}^d$ and
$\tbold^{\boldsymbol{\alpha}} =
\prod_{j=1}^d T_j^{\alpha_j}$ for
$\boldsymbol{\alpha} =
(\alpha_j)_{j=1}^d \in \zbb_+^d$. The
Taylor spectrum of $\tbold$ is denoted
by $\sigma(\tbold)$ (see \cite{Tay70}).
We say that $\lambdab =
(\lambda_j)_{j=1}^d \in \cbb^d$ is a
{\em joint eigenvalue} of $\tbold$ if
$\bigcap_{j=1}^d \nul(\lambda_j I-T_j)
\neq \{0\}$. The space $\bigcap_{j=1}^d
\nul(\lambda_j I-T_j)$ is called the
{\em joint eigenspace} corresponding to
the joint eigenvalue $\lambdab$. The
set of all joint eigenvalues of
$\tbold$ is denoted by
$\sigma_{\mathsf{p}}(\tbold)$.

A $d$-tuple $\tbold=(T_j)_{j=1}^d \in
\ogr(\hh)^d$ is called {\em normal} if
each component $T_j$ is normal and the
components commute with each other.
Recall that for any normal $d$-tuple
$\tbold=(T_j)_{j=1}^d \in \ogr(\hh)^d$,
there exists a unique spectral
measure\footnote{\;$\borel(X)$ denotes
the $\sigma$-algebra of all Borel
subsets of a topological Hausdorff
space $X$.} $E_{\tbold} \colon
\borel(\cbb^d) \to \ogr(\hh)$, called
the {\em joint spectral measure} of
$\tbold$, such that (see
\cite[Theorem~5.21]{Sch12})
   \begin{align} \label{tyjyt}
T_j = \int_{\cbb^d} \xi_j \D E_{\tbold},
\quad j=1, \ldots, d,
   \end{align}
where $\xi_j$ is a complex function on
$\cbb^d$ defined by
   \begin{align} \label{corge}
\text{$\xi_j(\zbold)=z_j$ for $\zbold=(z_1,
\ldots, z_d)\in \cbb^d$ and $j=1, \ldots,
d$.}
   \end{align}
Moreover, the following identity holds
(cf.~\cite[Theorem~2.1]{C-J-J-S}):
   \begin{align} \label{dzer}
\sigma(\tbold)=\supp(E_{\tbold}),
   \end{align}
where $\supp(E_{\tbold})$ denotes the
closed support of $E_{\tbold}$. The
Stone-von Neumann calculus for a normal
$d$-tuple $\tbold=(T_j)_{j=1}^d \in
\ogr(\hh)^d$ is defined by
   \begin{align*}
f(\tbold)=\int_{\sigma(\tbold)}f\D
E_{\tbold}, \quad f\in C(X),
   \end{align*}
whenever $X$ is a compact subset
$\cbb^d$ such that $\sigma(\tbold)
\subseteq X$. Note that $f(\tbold)$
depends only on the restriction of $f$
to $\sigma(\tbold)$. The relationship
between representations of $C(X)$,
where $X\subseteq \cbb^d$ is compact,
and normal $d$-tuples is established in
Appendix~\ref{App.B}, in which the
notion of ``the representation induced
by a normal $d$-tuple'' is also
defined. For further background on the
spectral theorem and the Stone-von
Neumann calculus needed in this
article, we refer the reader to
\cite{Rud73,Weid80,Sch12}.
   \section{\label{Sec.5}Proofs of
Theorems~\ref{soorh} and~\ref{soorh-2}, and dichotomy
($\beta$) revisited.}
   We begin by describing completely positive maps
through UCP maps.
   \begin{lemma} 
   Let $\ascr$ be a $C^*$-algebra with
unit $e$, $\hh$ be a Hilbert space and
$\varPsi\colon \ascr \to \ogr(\hh)$ be
a map. Then the following conditions
are equivalent{\em :}
   \begin{enumerate}
   \item[(i)] $\varPsi$ is completely
positive,
   \item[(ii)] there exist a UCP map $\varPhi\colon \ascr \to
\ogr(\hh)$ and a completely positive map
$\widetilde\varPsi\colon \ogr(\hh) \to \ogr(\hh)$ such
that
   \begin{align} \label{vvpy}
\varPsi = \widetilde \varPsi \circ \varPhi,
   \end{align}
   \item[(iii)] there exist a UCP map $\varPhi\colon \ascr \to
\ogr(\hh)$ and $R\in \ogr(\hh)$ such that $R\Ge 0$~and
   \begin{align*}
\varPsi(a) = R \varPhi(a) R, \quad a\in \ascr.
   \end{align*}
   \end{enumerate}
Moreover, if {\em (iii)} holds, then
$R=\varPsi(e)^{1/2}$.
   \end{lemma}
   \begin{proof}
(i)$\Rightarrow$(iii) By the Stinespring dilation
theorem (see \cite[Theorem~1]{Sti55}), there exist a
Hilbert space $\kk$, an operator $B\in \ogr(\hh,\kk)$
and a representation $\pi\colon \ascr \to \ogr(\kk)$
such that
   \allowdisplaybreaks
   \begin{align} \label{vsi}
\varPsi(a) = B^*\pi(a)B, \quad a\in \ascr.
   \end{align}
Since $\varPsi$ is positive, $\varPsi(e)\Ge 0$. We may
assume without loss of generality that
   \allowdisplaybreaks
   \begin{align} \label{dxin}
\dim \overline{\ran(\varPsi(e))}^{\perp} \Le \dim
\overline{\ran(B)}^{\perp}.
   \end{align}
Indeed, there exists a cardinal number $\frak n\Ge 1$
such that $\frak n \cdot \dim \kk \Ge \dim \hh$. Set
$\mcal = \bigoplus_{\omega \in \varSigma}
\kk_{\omega}$ and $\kk_{\mcal}=\kk \oplus \mcal$,
where $\varSigma$ is a set of cardinality $\mathfrak
n$ and $\kk_{\omega}=\kk$ for every $\omega\in
\varSigma$. Clearly, $\dim \mcal = \frak n \cdot \dim
\kk$. Define the operator $B_{\mcal}\in
\ogr(\hh,\kk_{\mcal})$ and the representation
$\pi_{\mcal}\colon \ascr \to \ogr(\kk_{\mcal})$ by
$B_{\mcal} h=Bh \oplus 0$ for $h\in \hh$ and
$\pi_{\mcal} =\pi\oplus \bigoplus_{\omega \in
\varSigma} \pi_{\omega}$ with $\pi_{\omega}=\pi$. Then
   \begin{align*}
\dim \overline{\ran(\varPsi(e))}^{\perp} \Le \dim
\big(\overline{\ran(B)}^{\perp} \oplus \mcal\big) =
\dim \big(\overline{\ran(B)}\oplus \{0\}\big)^{\perp}
= \dim \overline{\ran(B_{\mcal})}^{\perp}.
   \end{align*}
Clearly \eqref{vsi} is valid with $\pi_{\mcal}$ and
$B_{\mcal}$ in place of $\pi$ and $B$, respectively.

Assume that \eqref{dxin} holds. Substituting $a=e$
into \eqref{vsi}, we get
   \begin{align} \label{bhyr}
\|Bh\|^2=\|\varPsi(e)^{1/2}h\|^2, \quad h\in \hh.
   \end{align}
Since $\overline{\ran(\varPsi(e)^{1/2})} =
\overline{\ran(\varPsi(e))}$, we deduce from
\eqref{bhyr} that there exists a unique unitary
operator $\widetilde U\in
\ogr(\overline{\ran(\varPsi(e))}, \overline{\ran(B)})$
such that $Bh=\widetilde{U} \varPsi(e)^{1/2}h$ for
every $h \in \hh$. Hence, by \eqref{dxin} there exist
an isometry $U\in \ogr(\hh,\kk)$ such that
   \begin{align} \label{psypas}
B=U\varPsi(e)^{1/2}.
   \end{align}
Define the map $\varPhi\colon \ascr \to \ogr(\hh)$ by
$\varPhi(a)=U^*\pi(a)U$ for $a\in \ascr$. The map
$\varPhi$, being a composition of the representation
$\pi$ and the UCP map $\ogr(\kk)\ni Y \mapsto U^*YU\in
\ogr(\hh)$, is a UCP map. By \eqref{vsi} and
\eqref{psypas} we have
   \begin{align*}
\varPsi(a) = B^*\pi(a)B = \varPsi(e)^{1/2}
\varPhi(a)\varPsi(e)^{1/2}, \quad a \in \ascr.
   \end{align*}
This means that (iii) holds with $R=\varPsi(e)^{1/2}$.

(iii)$\Rightarrow$(ii) Let $\varPhi$ and $R$ be as in
(iii). Then the map $\widetilde \varPsi$ defined by
$\widetilde \varPsi(Y) = RYR$ for $Y\in \ogr(\hh)$
satisfies \eqref{vvpy}.

(ii)$\Rightarrow$(i) This implication is obvious.
   \end{proof}
   \begin{proof}[Proof of Theorem~\ref{soorh}]
(i) Suppose that $G \subseteq \ker \phi$ for some
state $\phi$ on $\ascr$. Define the map $\varPsi_0
\colon \ascr \to \ogr(\hh)$ by $\varPsi_0(a)= \phi(a)
\triangle_{R}$ for $a \in \ascr$. Since states on
$C^*$-algebras are UCP maps (see
\cite[Proposition~3.8]{Paul02}), we infer from
\cite[Lemma~3.10]{Paul02} that the map $\varPsi_0$ is
completely positive. Let us define the maps
$\varPsi_1, \varPhi \colon \ascr \to \ogr(\hh)$ by
$\varPsi_1(a)= R^*\rho(a)R$ and $\varPhi(a) =
\varPsi_0(a) + \varPsi_1(a)$ for $a \in \ascr$.
Clearly, the map $\varPsi_1$ is completely positive
(see \cite[Theorem~1]{Sti55}). Hence, since
$\varPsi_0(e)=\triangle_{R}$, it follows that
$\varPhi$ is a UCP map ($e$ is the unit of $\ascr$).
As $G \subseteq \ker\phi$, we have $\pi(g) =
\varPhi(g)$ for all $g\in G$. Thus, by the unique
extension property (see \cite[Theorem~2.1]{Arv11}),
$\pi=\varPhi$,~i.e.,
   \begin{align} \label{pipu}
\pi(a)= \varPsi_0(a) + \varPsi_1(a), \quad a \in
\ascr.
   \end{align}

Now we prove that $R$ intertwines $\pi$ and $\rho$.
Take finite sequences $\{a_j\}_{j=1}^n \subseteq
\ascr$ and $\{h_j\}_{j=1}^n \subseteq \hh$. Since
$\varPsi_0$ is completely positive and the matrix
$[a_k^{*}a_j]_{k,j=1}^n$ is positive, we deduce that
$\sum_{j,k=1}^n \is{\varPsi_0(a_k^{*}a_j)h_j}{h_k} \Ge
0$. This together with \eqref{pipu}~yields
   \allowdisplaybreaks
   \begin{align*}
\bigg\|\sum_{j=1}^n \pi(a_j)h_j\bigg\|^2 & =
\sum_{j,k=1}^n \is{\pi(a_k^{*}a_j)h_j}{h_k}
   \\ \notag
& = \sum_{j,k=1}^n \is{\rho(a_k^{*}a_j)R h_j}{Rh_k} +
\sum_{j,k=1}^n \is{\varPsi_0(a_k^{*}a_j)h_j}{h_k}
   \\
& \Ge \sum_{j,k=1}^n \is{\rho(a_k^{*}a_j)R h_j}{Rh_k}
   \\
& = \bigg\|\sum_{j=1}^n \rho(a_j)Rh_j\bigg\|^2.
   \end{align*}
Since $\pi(e)=I_{\hh}$, the vectors
$\big\{\pi(a)h\colon a \in \ascr, \, h\in \hh\big\}$
span $\hh$. As a consequence, there exists a unique
contraction $\widehat R\in \ogr(\hh,\kk)$ such that
$\widehat R\pi(a)=\rho(a)R$ for every $a\in \ascr$.
Substituting $a=e$, we see that $\widehat R=R$.
Therefore, we have
   \begin{align} \label{prypliv}
R \pi(a)=\rho(a) R, \quad a\in \ascr.
   \end{align}

Our next goal is to show that $J$ is a
closed ideal in $\ascr$. Clearly, $J$
is a closed selfadjoint vector subspace
of $\ascr$. It follows from
\eqref{prypliv} that
$R^*\rho(a)=\pi(a)R^*$ for every $a\in
\ascr$, which implies that for all
$a\in J$ and $b,c\in \ascr$,
   \begin{align} \notag
\pi(bac)=\pi(b)\pi(a) \pi(c) & = \pi(b) R^*\rho(a)R
\pi(c)
   \\ \label{ifyl}
& = R^*\rho(b)\rho(a)\rho(c) R = R^*\rho(bac)R.
   \end{align}
This shows that $J$ is a closed proper ideal in
$\ascr$ (since $R$ is non-isometric). By \eqref{pipu}
we have $\ker \phi \subseteq J$. As $\ker \phi$ has
codimension $1$ in $\ascr$, it follows that $\ker \phi
= J$, and consequently $\phi$ is a character of
$\ascr$. The identity \eqref{xzaw} follows from
\eqref{pipu}, which completes the proof of (i) and of
the ``moreover'' part.

(ii) ({\sc The first proof}) Assume that for every
state $\phi$ on $\ascr$, $G\nsubseteq \ker \phi$. We
proceed by contradiction. If $R\pi(a) = \rho(a) R$ for
all $a \in \ascr$, then arguing as in \eqref{ifyl}, we
deduce that $J$ is an ideal. By assumption, the set
$G$ is finite or countably infinite, so by rescaling
there is no loss of generality in assuming that the
series $\sum_{g\in G} \|g\|^2$ is convergent. Clearly,
$a := \sum_{g\in G} g^*g$ is selfadjoint and $a\in J$.
Denote by $\mathscr{S}(\ascr)$ the compact space of
all states of $\ascr$, equipped with the weak$^*$
topology. A straightforward argument shows that the
map $\mathscr{S}(\ascr) \ni \phi \to \sum_{g\in G}
|\phi(g)|^2 \in \rbb_+$ is continuous, and by
assumption it takes only positive values. So by the
Weierstrass theorem there exists $\varepsilon > 0$
such that $\sum_{g\in G} |\phi(g)|^2 \Ge \varepsilon$
for every $\phi \in \mathscr{S}(\ascr)$. According to
the Cauchy-Schwarz inequality for states, we obtain
   \begin{align*}
\phi(a) = \sum_{g\in G} \phi(g^*g) \Ge \sum_{g\in G}
|\phi(g)|^2 \Ge \phi(\varepsilon e), \quad \phi\in
\mathscr{S}(\ascr).
   \end{align*}
It follows from \cite[Theorem~3.4.3]{Mur90} that $a -
\varepsilon e \Ge 0$, hence $a$ is invertible in
$\ascr$. Since $a \in J$, we deduce that $e\in J$, or
equivalently, that $R$ is an isometry, a
contradiction.

(ii) ({\sc The second proof}) 
As above, we proceed by contradiction, arriving at the
statement that $J$ is a proper closed ideal in
$\ascr$. Denote by $\ascr_G$ the $*$-algebra (not
necessarily unital) generated by $G$. By assumption
$G\subseteq J$, so $\ascr_{G} \subseteq J$. We claim
that $\ascr = J \dotplus \cbb e$. Take $a\in \ascr$.
Since $G$ generates the unital $C^*$-algebra $\ascr$,
there exist sequences $\{b_n\}_{n=1}^{\infty}
\subseteq \ascr_{G}\subseteq J$ and
$\{\lambda_n\}_{n=1}^{\infty}\subseteq \cbb$ such that
$a_n:=b_n + \lambda_n e \to a$ as $n\to\infty$ ($e$
stands for the unit of $\ascr$). First, consider the
case when the sequence $\{\lambda_n\}_{n=1}^{\infty}$
is unbounded. Passing to a subsequence, if necessary,
we can assume without loss of generality that
$\lim_{n\to\infty}|\lambda_n|=\infty$. Then
$\lim_{n\to\infty} \frac{1}{\lambda_n} a_n =0$, which
yields $e=-\lim_{n\to\infty} \frac{1}{\lambda_n} b_n
\in J$, so the ideal $J$ is not proper, a
contradiction. Therefore, the sequence
$\{\lambda_n\}_{n=1}^{\infty}$ is bounded. In view of
Bolzano-Weierstrass theorem, passing to a subsequence
if necessary, we can assume without loss of generality
that $\lim_{n\to\infty} \lambda_n =\lambda$ for some
$\lambda \in \cbb$. Then the sequence
$\{b_n\}_{n=1}^{\infty}$ is convergent and
$a=b+\lambda e$, where $b=\lim_{n\to\infty} b_n \in
J$. This shows that $\ascr = J \dotplus \cbb e$.
Hence, there exists a unique character $\chi$ of
$\ascr$ such that $G \subseteq J = \ker\chi$, a
contradiction. This completes the proof.
   \end{proof}
   \begin{proof}[Proof of Theorem~\ref{soorh-2}]
It follows from Theorem~\ref{soorh}($\alpha$) that
   \begin{align} \label{rprs}
\text{$R \pi(a) = \rho(a) R$ and $R^*\rho(a)= \pi(a)
R^*$ for all $a\in \ascr$.}
   \end{align}

(i) By \eqref{rprs}, we have
   \begin{align} \label{cykot}
R^*R \pi(a) = R^* \rho(a) R = \pi(a) R^*R, \quad a \in
\ascr,
   \end{align}
so $\triangle_{R}$ commutes with $\pi$. Similarly,
$\triangle_{R^*}$ commutes with $\rho$ because
   \begin{align} \label{rrsa}
\rho(a) RR^* = R \pi(a) R^* = RR^* \rho(a), \quad a
\in \ascr,
   \end{align}
which completes the proof of (i).

(ii)\&(iii) In view of (i) and \eqref{rprs},
$\overline{\ran(\triangle_{R})}$ reduces $\pi$ and
$\overline{\ran(\triangle_{R^*}R)}$ reduces $\rho$
(note that $\overline{\ran(\triangle_{R})}\neq \{0\}$
since $R$ is non-isometric). Since $R\triangle_{R} =
\triangle_{R^*} R$, we get $R(\ran(\triangle_{R}))
\subseteq \ran(\triangle_{R^*} R)$, so $R(\hh_1)
\subseteq \kk_1$. In turn, if $h\in
\hh_0=\ker\triangle_{R}$, then for every $g\in \hh$,
   \begin{align*}
\is{Rh}{\triangle_{R^*} Rg} = \is{Rh}{ R\triangle_{R}
g} = \is{R^*Rh}{\triangle_{R} g} =
\is{h}{\triangle_{R} g} = \is{\triangle_{R} h}{g} = 0,
   \end{align*}
which implies that $Rh \in \kk_1^{\perp} = \kk_0$.
Hence $R(\hh_0) \subseteq \kk_0$. This yields $R=R_1
\oplus R_0$, where $R_j\in \ogr(\hh_j, \kk_j)$ for
$j=0,1$ are given by $R_jh=Rh$ for $h\in \hh_j$ and
$j=0,1$. That $R_0$ is isometric follows from the
equality $\hh_0=\ker\triangle_{R}$. If $h\in \hh_1$
and $\|R_1 h\|=\|h\|$, then $\is{\triangle_{R}
h}{h}=0$, which together with $\triangle_{R} \Ge 0$
implies that $h\in \hh_0$, so $h=0$. Thus $R_1$ is a
pure contraction. Since $R=R_1 \oplus R_0$ and
$\triangle_{R}|_{\hh_0}=0$, we get
   \begin{align} \label{dyzi}
\triangle_{R^*} R = R \triangle_{R} = R_0
\triangle_{R}|_{\hh_0} \oplus R_1
\triangle_{R}|_{\hh_1} = 0 \oplus R_1
\triangle_{R}|_{\hh_1}.
   \end{align}
However, $\ker(\triangle_{R}|_{\hh_1})=\{0\}$, so
$\overline{\ran{(\triangle_{R}|_{\hh_1})}}=\hh_1$,
which together with \eqref{dyzi} yields
   \begin{align*}
\kk_1=\{0\} \oplus \overline{R_1(
\overline{\ran{(\triangle_{R}|_{\hh_1}})})} = \{0\}
\oplus \overline{\ran(R_1)}.
   \end{align*}
As a consequence, $R_1$ has dense range.

It follows from \eqref{xzaw} and \eqref{rprs} that
   \begin{align} \label{wyledil}
\pi(a) = R^*\rho(a)R + \chi(a) \triangle_{R} =
\pi(a)R^*R + \chi(a) \triangle_{R}, \quad a \in \ascr.
   \end{align}
This implies that $\pi(a) \triangle_{R} = \chi(a)
\triangle_{R}$ for all $a\in \ascr$. Therefore,
$\pi(a)|_{\hh_1} = \chi(a) I_{\hh_1}$ for every $a \in
\ascr$. In turn, by \eqref{rprs}, \eqref{wyledil} and
(i), we have
   \begin{align*}
R^* \rho(a) = \pi(a) R^* & = R^* \rho(a) RR^* +
\chi(a)
\triangle_{R} R^* \\
&= R^*RR^* \rho(a) + \chi(a) \triangle_{R} R^*, \quad
a \in \ascr.
   \end{align*}
This yields
   \begin{align*}
R^*\triangle_{R^*} \rho(a) = \chi(a) \triangle_{R} R^*
= \chi(a) R^*\triangle_{R^*}, \quad a \in \ascr.
   \end{align*}
Taking adjoints, we see that
   \begin{align*}
\rho(a) \triangle_{R^*} R = \chi(a) \triangle_{R^*} R,
\quad a \in \ascr.
   \end{align*}
This implies that $\rho(a)|_{\kk_1} = \chi(a)
I_{\kk_1}$ for every $a\in \ascr$. Since $R_0$ is
isometric, we get
   \begin{align*}
\pi(a)|_{\hh_0} & \overset{\eqref{wyledil}}= R_0^*
\rho(a)|_{\kk_0} R _0 \oplus \chi(a) \triangle_{R_0} =
R_0^*\rho(a)|_{\kk_0}R _0, \quad a \in \ascr.
   \end{align*}
This completes the proof.
   \end{proof}
{\em Vice versa} (concerning Theorem~\ref{soorh-2}),
the following holds.
   \begin{prop} \label{ghkds}
Let $G$ be a nonempty subset of a unital $C^*$-algebra
$\ascr$, and let $\chi$ be a character of $\ascr$ such
that $G\subseteq \ker \chi$. Suppose that $R\in
\ogr(\hh,\kk)$ is a contraction of the form $R=R_1
\oplus R_0$ with $R_j\in \ogr(\hh_j,\kk_j)$ for
$j=0,1$, $\pi= \chi I_{\hh_1}\oplus \pi_0 $, $\rho=
\chi I_{\kk_1} \oplus \rho_0 $ and
$\pi_0=R_0^*\rho_0R_0$, where $\pi_0\colon \ascr \to
\ogr(\hh_0)$ and $\rho_0\colon \ascr \to \ogr(\kk_0)$
are representations. Then $\pi(g) = R^*\rho(g)R$ for
all $g\in G$ and $R_0$ is an isometry. Moreover, if
$R$ $($equivalently $R_1$$)$ is non-isometric, then
   \begin{enumerate}
   \item[(i)] $\ker \chi = J$  and
condition {\em (ii)} of Theorem~{\em \ref{soorh-2}}
holds,
   \item[(ii)]
$\hh_1 = \overline{\ran(\triangle_{R})}$ provided
$R_1$ is a pure contraction,
   \item[(iii)] $\kk_1=\overline
{\ran(\triangle_{R^*}R)}$ provided $R_1$ is a pure
contraction with dense range.
   \end{enumerate}
   \end{prop}
   \begin{proof}
Since $G\subseteq \ker \chi$, it is routine to verify
that $\pi(g) = R^*\rho(g)R$ for all $g\in G$. Taking
the values of both sides of the equality
$\pi_0=R_0^*\rho_0R_0$ at the unit of $\ascr$, we see
that $R_0$ is an isometry.

To prove the ``moreover'' part, assume that $R_1$ is
non-isometric.

(i) Note that $a\in \ascr$ is in $J$ if and only if
   \begin{align*}
\chi(a) I_{\hh_1} \oplus R_0^*\rho_0(a) R_0 & =
\chi(a) I_{\hh_1} \oplus \pi_0(a)
   \\
& = \pi(a) = R^*\rho(a)R = \chi(a) R_1^*R_1 \oplus
R_0^* \rho_0(a) R_0,
   \end{align*}
or equivalently if and only if
$\chi(a)(I_{\hh_1}-R_1^*R_1)=0$. Since $R_1$ is
non-isometric, we conclude that $J=\ker\chi$.

Now we show that $\triangle_{R^*}$ commutes with
$\rho$. To this end, observe that
   \begin{align*}
(R_0^* \rho_0(a)R_0)^* (R_0^* \rho_0(a)R_0) & =
\pi_0(a)^* \pi_0(a)
   \\
& = \pi_0(a^*a) = R_0^* \rho_0(a)^*\rho_0(a) R_0,
\quad a \in \ascr.
   \end{align*}
By \cite[Lemma~3.2]{P-S-roots23}, $\rho_0(a) R_0 =
R_0R_0^* \rho_0(a) R_0$ for all $a\in \ascr$.
Multiplying by $R_0^*$ on the right side, we see that
$\rho_0(a) R_0R_0^* = R_0R_0^* \rho_0(a) R_0R_0^*$ for
all $a \in \ascr$. Taking adjoints, we obtain
$R_0R_0^* \rho_0(a) = \rho_0(a) R_0R_0^*$ for all
$a\in \ascr$. Therefore, we have
   \begin{align*}
RR^* \rho(a) & = \chi(a) R_1R_1^* \oplus R_0R_0^*
\rho_0(a)
   \\
&= \chi(a) R_1R_1^* \oplus \rho_0 (a) R_0R_0^*
=\rho(a)RR^*, \quad a \in \ascr,
   \end{align*}
so $\triangle_{R^*}$ commutes with $\rho$. This
implies that
   \begin{align*}
\pi_0(a) R_0^*R_0 = R_0^* \rho_0(a) (R_0 R_0^*)R_0 =
R_0^* R_0 (R_0^*\rho_0(a) R_0) = R_0^* R_0 \pi_0(a),
\quad a \in \ascr.
   \end{align*}
Arguing as above, we conclude that $\triangle_{R}$
commutes with $\pi$.

(ii) Since $R_0$ is an isometry, we get
   \begin{align} \label{frwad}
\triangle_{R} = (I_{\hh_1} - R_1^*R_1) \oplus
(I_{\hh_0} - R_0^*R_0) = (I_{\hh_1} - R_1^*R_1) \oplus
0.
   \end{align}
However, $R_1$ is a pure contraction, so
$\ker(I_{\hh_1} - R_1^*R_1)=\{0\}$ and consequently
$\overline{\ran(I_{\hh_1}-R_1^*R_1)} = \hh_1$, which
together with \eqref{frwad} yields
$\overline{\ran(\triangle_{R})}= \hh_1$.

(iii) Arguing as in (ii), we obtain
   \begin{align*}
\triangle_{R^*} R = R\triangle_{R} = R_1
(I_{\hh_1}-R_1^*R_1) \oplus R_0 (I_{\hh_0}-R_0^*R_0) =
R_1 (I_{\hh_1}-R_1^*R_1) \oplus 0,
   \end{align*}
which implies that
   \begin{align*}
\overline{\ran(\triangle_{R^*} R)}= \overline{R_1(
\overline{\ran{(I_{\hh_1}-R_1^*R_1)}})} \oplus \{0\} =
\overline{\ran(R_1)} \oplus \{0\} = \kk_1.
   \end{align*}
This completes the proof.
   \end{proof}
The following result, largely implicit in the proofs
of Theorems~\ref{soorh} and~\ref{soorh-2}, sheds
additional light on the connection between the
intertwining relation and the ideal property of $J$.
   \begin{lemma} \label{pirhr}
Let $\ascr$ be a unital $C^*$-algebra $\ascr$. Suppose
that $\hh$ and $\kk$ are Hilbert spaces, $\pi\colon
\ascr \to \ogr(\hh)$ and $\rho\colon \ascr \to
\ogr(\kk)$ are representations and $R\in
\ogr(\hh,\kk)$ is a contraction. Set $\varPsi(a)=
R^*\rho(a)R$ for $a\in \ascr$. Then the set $J$
defined in \eqref{zbyr} is a selfadjoint vector space,
and the following assertions hold{\em :}
   \begin{enumerate}
   \item[(i)] if $R \pi(a) = \rho(a) R$
for all $a \in \ascr$, then $J$ is an ideal,
   \item[(ii)] if $A\subseteq \ascr$ is selfadjoint, and $R
\pi(a) = \rho(a) R$ for all $a \in A$, then $R^*R$
commutes with $\pi$, $RR^*$ commutes with $\rho$, and
$R^*R \pi(a) = R^* \rho(a) R$ for all $a \in A$,
   \item[(iii)] if $J$ is an ideal, then
$R \pi(a) = \rho(a) R$, $\triangle_R \pi(a) = \pi(a)
\triangle_R = 0$ for all $a \in J$.
   \end{enumerate}
   \end{lemma}
   \begin{proof}
(i) See the reasoning carried out in
\eqref{ifyl}.

(ii) Proceed as in \eqref{cykot} and \eqref{rrsa}.

(iii) Take $a \in J$. Then $a^*a \in J$, so by
\cite[Lemma~3.2]{P-S-roots23} we obtain
   \begin{align*}
\pi(a^*a) = (R^*
\rho(a) R)^*(R^* \rho(a) R) \Le R^*
\rho(a)^* \rho(a) R =
\pi(a^*a).
   \end{align*}
Hence, by \cite[Lemma~3.2]{P-S-roots23} again,
$\rho(a)R= R(R^*\rho(a)R)=R\pi(a)$, and therefore,
since $a^* \in J$, it follows that $R^*\rho(a) =
\pi(a) R^*$. Consequently,
   \begin{align*}
\pi(a) = R^* \rho(a) R = R^* R \pi(a)
\overset{\mathrm{(ii)}}= \pi(a)R^* R,
   \end{align*}
which shows that (iii) holds.
   \end{proof}
   \begin{rem}
The set $J$ is closely related to the so-called {\em
multiplicative domain} $MD(\varPsi)$ of a CCCP map
$\varPsi\colon \ascr\to \ogr(\hh)$, which is defined
by (see \cite{Cho-74})
   \begin{align*}
MD(\varPsi)=\big\{a\in \ascr\colon
\varPsi(a^*a)=\varPsi(a)^*\varPsi(a) \;\text{ and }\;
\varPsi(aa^*)=\varPsi(a)\varPsi(a)^* \big\}.
   \end{align*}
Suppose that the triple $(\pi,\rho,R)$ satisfies the
assumptions of Lemma~\ref{pirhr}, that $\varPsi(a)=
R^*\rho(a)R$ for all $a\in \ascr$, and that $J$ is an
ideal. Since $J= \{a \in \ascr\colon
\pi(a)=\varPsi(a)\}$, it follows that $J \subseteq
MD(\varPsi)$. If, in addition, the quadruple
$(G,\pi,\rho,R)$ satisfies the assumptions of
Theorem~\ref{soorh} and the hyperrigid set $G$ is as
in alternative $(\alpha)$ of the dichotomy, then we
have two cases: either $J=MD(\varPsi)$ when $R$ is not
a partial isometry, or $J\subsetneq MD(\varPsi)=\ascr$
otherwise.
   \hfill $\diamondsuit$
   \end{rem}
Below, $\sigma(a)$ denotes the spectrum of an element
$a$ of a unital $C^*$-algebra.
   \begin{prop} \label{glmp}
Let $G$ be a hyperrigid subset of a unital
$C^*$-algebra $\ascr$ with unit $e$, and let
$\lambdab=\{\lambda_g\}_{g\in G} \subseteq \cbb$. Then
$G_{\lambdab}:=\{g-\lambda_g e\colon g\in G\}$ is
hyperrigid in $\ascr$. If, in addition, $\lambda_{g_0}
\in \cbb \setminus \sigma(g_0)$ for some $g_0\in G$,
$\hh$ and $\kk$ are separable Hilbert spaces,
$\pi\colon \ascr \to \ogr(\hh)$ and $\rho\colon \ascr
\to \ogr(\kk)$ are representations, and $R\in
\ogr(\hh,\kk)$ is a non-isometric contraction such
that $\pi(g) = R^*\rho(g)R$ for all $g\in
G_{\lambdab}$, then
   \begin{enumerate}
  \item[(i)] $J$ is not
an ideal,
  \item[(ii)] the intertwining relation
$R \pi=\rho R$ $($see
\eqref{nicuntr}$)$ does not hold.
   \end{enumerate}
   \end{prop}
   \begin{proof}
Using the unique extension property, we deduce that
$G_{\lambdab}$ is hyperrigid.

(i) Assume, for contradiction, that $J$ is an ideal.
Then, by Lemma~\ref{pirhr}(iii), we have $R^*R \pi(a)
= \pi(a)$ for all $a \in J$. Since $G_{\lambdab}
\subseteq J$, it follows that
\begin{align} \label{rsgo}
R^*R \big(\pi(g_0) - \lambda_{g_0} I\big) = \pi(g_0) -
\lambda_{g_0} I.
\end{align}
But $\sigma(\pi(g_0)) \subseteq \sigma(g_0)$ and
$\lambda_{g_0} \notin \sigma(g_0)$, so $\pi(g_0) -
\lambda_{g_0} I$ is invertible. In view of
\eqref{rsgo}, this gives $R^*R = I$, a contradiction.

(ii) This follows from Lemma~\ref{pirhr}(i) and part
(i) above.
   \end{proof}
   \begin{rem}
If we replace the requirement ``$\lambda_{g_0} \in
\cbb \setminus \sigma(g_0)$'' with the condition ``for
every state $\phi$ on $\ascr$, $\phi(g_0) \neq
\lambda_{g_0}$,'' or equivalently, that
$\lambda_{g_0}$ is not in the algebraic numerical
range $W(g_0)$ of $g_0$, then the conclusion of
Proposition~\ref{glmp} follows from
Theorem~\ref{soorh}($\beta$). On the other hand, by
\cite[Proposition~4.3.3]{Kad-Rin83} we have
$\sigma(g_0) \subseteq W(g_0)$, and hence
$\lambda_{g_0} \in \cbb \setminus \sigma(g_0)$, so we
are exactly in the situation of
Proposition~\ref{glmp}.
   \hfill $\diamondsuit$
   \end{rem}
   \section{\label{Sec.5.5}Proofs of Theorem~\ref{ogylu} and two of its extensions.}
Since the concept of a minimal $R$-dilation play a
role in the proofs of Theorems~\ref{ogylu}
and~\ref{rconrep} and Proposition~\ref{gfza}, we
summarize its main properties below. In particular,
note that in condition (ii) of Theorem \ref{ogylu}, we
may, without loss of generality, assume that the
representation $\rho$ is a minimal $R$-dilation of the
representation $\pi$, in the sense that $\bigvee_{a\in
\ascr} \rho(a)R(\hh)=\kk$.
   \begin{lemma} \label{deqa}
Suppose that $\pi\colon \ascr \to \ogr(\hh)$ and
$\rho\colon \ascr \to \ogr(\kk)$ are representations
of a unital $C^*$-algebra $\ascr$, and that $R \in
\ogr(\mathcal{H},\kk)$. Then
   \begin{enumerate}
   \item[(i)] $\kk_{\mathrm{m}} := \bigvee_{a\in \ascr}
\rho(a)R(\hh)$ reduces $\rho$ to a representation
$\rho_{\mathrm{m}}\colon \ascr \to
\ogr(\kk_{\mathrm{m}})$,
   \item[(ii)] $R^*\rho(a) R = R_{\mathrm{m}}^*\rho_{\mathrm{m}}(a)
R_{\mathrm{m}}$ for all $a\in \ascr$,
   \item[(iii)] $R_{\mathrm{m}} \pi(a) = \rho_{\mathrm{m}}(a)
R_{\mathrm{m}}$ for all $a\in \ascr$ whenever $R
\pi(a) = \rho(a) R$ for all $a\in \ascr$,
   \end{enumerate}
where $R_{\mathrm{m}} \in \ogr(\hh, \kk_{\mathrm{m}})$
is the unique contraction such that $R_{\mathrm{m}} h
= Rh$ for all $h\in \hh$.
   \end{lemma}
   \begin{proof}
Since $\rho$ is a representation, it is easy to see
that $\kk_{\mathrm{m}}$ reduces $\rho$, and that
$\ran(R) \subseteq \kk_{\mathrm{m}}$, which implies
that $R_{\mathrm{m}}$ is well-defined and that
$R^*\rho(a) R = R_{\mathrm{m}}^*\rho_{\mathrm{m}}(a)
R_{\mathrm{m}}$ for all $a\in \ascr$. The identity
$R_{\mathrm{m}} \pi = \rho_{\mathrm{m}}
R_{\mathrm{m}}$ is then obvious.
   \end{proof}
   \begin{proof}[Proof of Theorem~\ref{ogylu}]
(i)$\Rightarrow$(ii) 
Use Theorems~\ref{delt} and \ref{soorh}.

(ii)$\Rightarrow$(i) 
Let $\pi\colon \ascr \to \ogr(\hh)$ be a
representation on a separable Hilbert space $\hh$ and
$\widetilde \varPhi\colon \ascr \to \ogr(\hh)$ be a
UCP map such that $\widetilde \varPhi|_{G} =
\pi|_{G}$. It follows from the Stinespring dilation
theorem that
   \begin{align} \label{ryni}
\widetilde \varPhi(a)=P\rho(a)|_{\hh},
\quad a\in \ascr,
   \end{align}
where $\rho\colon \ascr \to \ogr(\kk)$ is a
representation on a Hilbert space $\kk$ such that
$\hh\subseteq \kk$, and $P\in \ogr(\kk)$ is the
orthogonal projection of $\kk$ onto $\hh$. There is no
loss of generality in assuming that $\kk=\bigvee_{a\in
\ascr} \rho(a)\hh$. Since $\ascr$ and $\hh$ are
separable, there exist countable sets
$\ascr_0\subseteq \ascr$ and $\hh_0 \subseteq \hh$
such that $\kk=\bigvee_{a\in \ascr_0} \rho(a)\hh_0$.
Hence $\kk$ is separable. Applying (ii) with $R\in
\ogr(\hh,\kk)$ given by $Rh=h$ for $h\in \hh$, we
deduce that $\pi(a)=\rho(a)|_{\hh}$ for all $a\in
\ascr$, and thus $\hh$ reduces $\rho$. Therefore, by
\eqref{ryni}, $\widetilde \varPhi$ is a representation
of $\ascr$. Since $G$ generates $\ascr$, it follows
from $\widetilde \varPhi|_{ G} = \pi|_{ G}$ that
$\widetilde \varPhi = \pi$. Thus $\pi|_{G}$ has the
unique extension property, and by
\cite[Theorem~2.1]{Arv11}, (i) holds.

(ii)$\Rightarrow$(iii) 
Assume $(\pi,\rho,R)$ is as in (iii) and
$\pi(g)=R^*\rho(g)R$ for all $g \in G$. If $R^*R=I$,
then by Theorem~\ref{delt} the implication
\eqref{ramc} holds. If $R$ is not isometric, then by
(ii)$\Rightarrow$(i) the set $G$ is hyperrigid, and,
by Theorem~\ref{soorh}, \eqref{ramc} holds as well.

(iii)$\Rightarrow$(ii) 
Assume $(\pi, \rho, R)$ is as in (ii) and
$\pi(g)=R^*\rho(g)R$ for all $g \in G$. By (iii), the
then-clause of \eqref{ramc} holds. First, consider the
case when $\pi(a)=R^*\rho(a)R$ for all $a \in \ascr$.
Then by Theorem~\ref{delt} the implication
\eqref{nicuntr} follows. The only other possibility is
that $R$ is not an isometry and
   \begin{align} \label{3pirar}
\pi(a)=R^*\rho(a)R, \quad a \in \ker \chi,
   \end{align}
where $\chi \in \mathfrak{M}_{\ascr}$. Because $\ker
\chi$ is an ideal of $\ascr$, we deduce from
\eqref{3pirar} that
   \begin{align*} 
(R^*\rho(a)R)^*(R^*\rho(a)R)=\pi(a)^*\pi(a)
=\pi(a^*a)=R^*\rho(a)^*\rho(a)R, \quad a \in \ker
\chi.
   \end{align*}
Hence, by \cite[Lemma~3.2]{P-S-roots23}, we have
   \begin{align*} 
\rho(a)R=R(R^*\rho(a)R)=R\pi(a), \quad a \in \ker
\chi.
   \end{align*}
Since $\ker \chi$ has codimension $1$ in $\ascr$, we
conclude that $R\pi(a)=\rho(a)R$ for all $a \in
\ascr$.

(iii)$\Leftrightarrow$(iv) Use the Stinespring
dilation theorem together with \eqref{stymyr}, noting
that the minimal dilation space $\kk = \bigvee_{a \in
\ascr} \rho(a)R(\hh)$ is separable whenever $\hh$ is.


To prove the ``moreover'' part, assume that $G$ is
hyperrigid and that the operator $R$ in (iii) is not
isometric. It follows from the ``moreover'' part of
Theorem~\ref{soorh} that $\phi$ is a character with
$\ker \chi \subseteq J = \ker \phi$. Hence $\ker \chi
= \ker \phi$, and therefore $\chi = \phi$. The case
$\varPsi(e) \neq I$ can be treated analogously.
   \end{proof}
   \begin{proof}[Proof of Corollary~\ref{caly}]
(i)$\Rightarrow$(ii) Let $(\pi,\rho,R)$ be as in (ii)
and $\pi(g) = R^*\rho(g)R$ for all $g\in G$. Since
$\ascr$ has no characters, implication
(i)$\Rightarrow$(iii) of Theorem~\ref{ogylu} shows
that $R$ is an isometry and the identity $\pi=R^*\rho
R$ holds on the whole algebra~$\ascr$.

(ii)$\Leftrightarrow$(iii) This equivalence is a
direct consequence of \eqref{stymyr} and the
Stinespring dilation theorem. In particular, by the
previous paragraph, $\varPsi$ is a UCP map.

(iii)$\Rightarrow$(i) In view of \eqref{stymyr} and
\cite[Theorem~2.1]{Arv11}, this implication is
obvious.
   \end{proof}
   \begin{rem} \label{qqre}
Inspection of the proof of Theorem~\ref{ogylu} shows
that the equivalences (ii)$\Leftrightarrow$(iii) and
(iii)$\Leftrightarrow$(iv), as well as the implication
(ii)$\Rightarrow$(i), hold without assuming $G
\subseteq \ker \phi$. In contrast, the proof of
(i)$\Rightarrow$(ii) does rely on this assumption. As
shown in Section~\ref{9.e}, there are hyperrigid sets
not annihilated by any state for which the
corresponding intertwining relation does not hold.
\hfill $\diamondsuit$
   \end{rem}
The next two results relax condition (ii) of
Theorem~\ref{ogylu}, yet still fully characterize
hyperrigidity.
   \begin{prop} \label{gfza}
Let $G$ be a finite or countably infinite set of
generators of a unital $C^*$-algebra $\ascr$, and let
$\chi$ be a character on $\ascr$ such that $G
\subseteq \ker\chi$. Then the following conditions are
equivalent{\em :}
   \begin{enumerate}
   \item[(i)] $G$ is hyperrigid,
   \item[(ii)]
for all separable Hilbert spaces $\hh$ and $\kk$, all
representations $\pi\colon \ascr \to \ogr(\hh)$ and
$\rho\colon \ascr \to \ogr(\kk)$, and every
contraction $R\in \ogr(\hh,\kk)$, if either
$\mcal_{\chi}(\pi)=\{0\}$ or $\pi(a)=\chi(a) I$ for
all $a\in \ascr$, then the implication
\eqref{nicuntr}~holds.
   \end{enumerate}
   \end{prop}
   \begin{proof}
It suffices to show that the present condition (ii)
implies condition (ii) of Theorem~\ref{ogylu}. Assume
$\hh$ and $\kk$ are separable Hilbert spaces,
$\pi\colon \ascr \to \ogr(\hh)$ and $\rho\colon \ascr
\to \ogr(\kk)$ are representations, and $R\in
\ogr(\hh,\kk)$ is a contraction such that
   \begin{align} \label{pigrh}
\pi(g) = R^*\rho(g)R, \quad g\in G.
   \end{align}
Set $\hh_0 = \mcal_{\chi}(\pi)^{\perp}$ and
$\hh_1=\mcal_{\chi}(\pi)$. Then, by Lemma~\ref{mcxi},
each $\hh_j$ reduces $\pi$ to the representation
$\pi_j\colon \ascr \to \ogr(\hh_j)$ such that
$\mcal_{\chi}(\pi_0)=\{0\}$,
$\mcal_{\chi}(\pi_1)=\mcal_{\chi}(\pi)$ and $\pi_1(a)
= \chi(a) I_{\hh_1}$ for all $a\in \ascr$. Define
$\kk_j=\bigvee_{a\in \ascr} \rho(a) R(\hh_j)$ for
$j=1,2$. Then $\kk_j$ reduces $\rho$ to the
representation $\rho_j$, and $R(\hh_j) \subseteq
\kk_j$ for $j=1,2$. This implies that for each
$j=1,2$, there exists a unique operator $R_j \in
\ogr(\hh_j, \kk_j)$ such that $R_j(f)=Rf$ for all
$f\in \hh_j$. According to \eqref{pigrh}, $\pi_j(g) =
R_j^* \rho_j(g) R_j$ for all $g\in G$. By assumption,
$R_j\pi_j(a)=\rho_j(a)R_j$ for all $a\in \ascr$ and
$j=1,2$. Hence
   \begin{align*}
R\pi(a)(f_0\oplus f_1) & = R_0\pi_0(a)f_0 +
R_1\pi_1(a)f_1 = \rho_0(a)R_0 f_0 + \rho_1(a)R_1 f_1
   \\
& = \rho(a) R (f_0\oplus f_1), \quad f_0\in \hh_0, \,
f_1\in \hh_1.
   \end{align*}
Since $\hh=\hh_0 \oplus \hh_1$, it follows that
$R\pi(a) = \rho(a) R$ for every $a\in \ascr$.
   \end{proof}
   \begin{prop} \label{jger}
Let $G$ be a finite or countably infinite set of
generators of a commutative unital $C^*$-algebra
$\ascr$, and let $\chi$ be a character of $\ascr$ such
that $G \subseteq \ker\chi$. If there exists $g_0\in
G$ such that $\chi'(g_0) > 0$ for every character
$\chi' \neq \chi$ and $\chi(g_0) \Ge 0$, then the
following conditions are equivalent{\em :}
   \begin{enumerate}
   \item[(i)] $G$ is hyperrigid,
   \item[(ii)]
for all separable Hilbert spaces $\hh$ and $\kk$, all
representations $\pi\colon \ascr \to \ogr(\hh)$ and
$\rho\colon \ascr \to \ogr(\kk)$, and every
contraction $R\in \ogr(\hh,\kk)$ such that
$\mcal_{\chi}(\pi)=\{0\}$, the implication
\eqref{nicuntr} holds.
   \end{enumerate}
   \end{prop}
   \begin{proof}
In view of Theorem~\ref{ogylu} and
Proposition~\ref{gfza}, it suffices to prove that if
$\pi\colon \ascr \to \ogr(\hh)$ and $\rho\colon \ascr
\to \ogr(\kk)$ are representations on separable
Hilbert spaces, and $R\in \ogr(\hh,\kk)$ is a
contraction such that $\pi(a)=\chi(a) I$ for all $a\in
\ascr$, and
   \begin{align} \label{aswq}
\pi(g) = R^*\rho(g)R, \quad g\in G,
   \end{align}
then $R\pi(a)=\rho(a)R$ for all $a\in \ascr$. First,
observe that there is no loss of generality in
assuming that $\kk$ is a minimal $R$-dilation of
$\pi$, that is, $\kk=\bigvee_{a\in \ascr}
\rho(a)R(\hh)$. Indeed, otherwise, by
Lemma~\ref{deqa}, the subspace
$\kk_{\mathrm{m}}:=\bigvee_{a\in \ascr} \rho(a)R(\hh)$
reduces $\rho$ to the representation
$\rho_{\mathrm{m}}$. Since $R(\hh)\subseteq
\kk_{\mathrm{m}}$, there exists a unique operator
$R_{\mathrm{m}}\in \ogr(\hh,\kk_{\mathrm{m}})$ such
that $Rh=R_{\mathrm{m}}h$ for all $h\in \hh$. Then, we
have
   \begin{gather} \label{qqs}
R \pi(a) h= R_{\mathrm{m}} \pi(a) h \text{ and }
\rho(a) R h = \rho_{\mathrm{m}}(a) R_{\mathrm{m}} h
\text{ for all } h \in \hh \text{ and }a\in \ascr,
   \\  \notag
R^*\rho(a)R = R_{\mathrm{m}}^* \rho_{\mathrm{m}}(a)
R_{\mathrm{m}} \text{ for all } a\in \ascr.
   \end{gather}
Hence, by \eqref{aswq}, we obtain $\pi(g) =
R_{\mathrm{m}}^*\rho_{\mathrm{m}}(g)R_{\mathrm{m}}$
for all $g\in G$. Assuming that (ii) holds in the case
of minimal $R$-dilations, we conclude that
$R_{\mathrm{m}}\pi(a)=\rho_{\mathrm{m}}(a)R_{\mathrm{m}}$
for all $a\in \ascr$. It then follows from \eqref{qqs}
that $R\pi(a)=\rho(a)R$ for all $a\in \ascr$.

So assume that $\kk_{\mathrm{m}}=\kk$. Then, by
assumption and \eqref{aswq}, we have
   \begin{align*}
0 = \chi(g_0) I = \pi(g_0)= R^* \rho(g_0) R.
   \end{align*}
Since by the Gelfand-Krein theorem $g_0 \Ge 0$, it
follows that $\rho(g_0) \Ge 0$. This yields
   \begin{align*}
0=\|\sqrt{\rho(g_0)}Rh\|^2 = \is{R^* \rho(g_0)R h}{h}
= 0, \quad h \in \hh,
   \end{align*}
which implies that $\rho(g_0)R(\hh)=\{0\}$. As a
consequence of the commutativity of $\ascr$ and the
multiplicativity of $\rho$, we obtain
   \begin{align*}
\rho(g_0)\big(\rho(a)Rh\big) =
\rho(a)\big(\rho(g_0)Rh\big) = 0, \quad h\in \hh, \, a
\in \ascr.
   \end{align*}
Since $\kk_{\mathrm{m}}=\kk$, we conclude that
$\rho(g_0)=0$. By the spectral theorem for
representations (see \cite[Theorem~12.22]{Rud73}, see
also \cite[Theorem~A.3]{P-S-24}), the following~holds:
   \begin{align*}
0 = \rho(g_0) = \int_{\mathfrak{M}_{\ascr}} \hat{g}_0
\, \D E_{\rho},
   \end{align*}
where $\hat{g}_0$ denotes the Gelfand transform of
$g_0$ (see Section~\ref{Sec.7}), and $E_{\rho}$ is the
spectral measure of $\rho$. Since $\hat{g}_0 \Ge 0$,
the above implies
   \begin{align*}
0=E_{\rho}(\hat{g}_0^{-1}((0,\infty))) =
E_{\rho}(\mathfrak{M}_{\ascr} \setminus \{\chi\}).
   \end{align*}
In other words, $E_{\rho}(\varDelta)=
\delta_{\chi}(\varDelta) I_{\kk}$ for all $\varDelta
\in \borel(\mathfrak{M}_{\ascr})$. It now follows that
   \begin{align*}
\rho(a) R = \int_{\mathfrak{M}_{\ascr}} \hat{a} \, \D
E_{\rho} R = \chi(a) R = R \pi(a), \quad a \in \ascr.
   \end{align*}
In conclusion, condition (ii) of
Proposition~\ref{gfza} is satisfied.
   \end{proof}
   \section{\label{Sect.8}Proofs of Theorems~\ref{rconrep}
and~\ref{rroiu}, and Corollary~\ref{uspc}.}
   Apart from the above-mentioned proofs, this section
also provides two consequences of
Theorem~\ref{rconrep}, namely Corollaries~\ref{tiply}
and~\ref{hpse}. Let $\mcal_{\chi}(\pi)$ be as in
\eqref{mycj}. Using the decomposition
$\ascr=\ker\chi+\mathbb{C}\cdot e$ (where $e$ is the
unit of $\ascr$), we obtain
   \allowdisplaybreaks
   \begin{align*}
\mcal_{\chi}(\pi) & = \bigcap_{a \in
\ker\chi} \nul(\pi(a) - \chi(a)I),
   \\
\mcal_{\chi}(\pi)^\perp & = \bigvee_{a
\in \ascr} \overline{\ran(\pi(a) -
\chi(a)I)}=\bigvee_{a \in \ker\chi}
\overline{\ran(\pi(a) - \chi(a)I)}.
   \end{align*}
   \begin{lemma} \label{mcxi}
Let $\ascr$ be a unital $C^*$-algebra,
$\pi\colon \mathscr{A} \to \ogr(\hh)$
be a representation on a Hilbert space
$\hh$, and $\chi$ be a character of
$\mathscr{A}$. Denote by $P_{\chi}$ the
orthogonal projection of $\hh$ onto
$\mcal_{\chi}(\pi)$. Then
   \begin{enumerate}
    \item[(i)] $\mcal_{\chi}(\pi)$ is the largest
closed subspace reducing $\pi$ such that
   \begin{align*}
\pi(a)|_{\mcal_{\chi}(\pi)} =
\chi(a)I_{\mcal_{\chi}(\pi)}, \quad a
\in \ascr,
   \end{align*}
    \item[(ii)] if $A\in \ogr(\hh)$ commutes with
$\pi$, then $\mcal_{\chi}(\pi)$ reduces
$A$,
    \item[(iii)] $P_{\chi} \in W^*(\pi)$, where
$W^*(\pi)$ stands for the von Neumann algebra
generated by $\{\pi(a)\colon a \in \ascr\}$.
   \end{enumerate}
   \end{lemma}
   \begin{proof}
(i) Let $f \in \mcal_{\chi}(\pi)$.
Then, for every $a \in \ascr$, we have
    \begin{align*}
\pi(b)(\pi(a)f) = \pi(ba)f = \chi(ba)f =
\chi(b)(\chi(a)f) = \chi(b)(\pi(a)f), \quad b\in
\ascr.
    \end{align*}
Hence, $\pi(a)\mcal_{\chi}(\pi)
\subseteq \mcal_{\chi}(\pi)$ for all
$a\in \ascr$, which implies that
$\mcal_{\chi}(\pi)$ reduces $\pi$. In
turn, if $\mcal$ is a closed subspace
of $\hh$ reducing $\pi$ and
$\pi(a)|_{\mcal}=\chi(a)I_{\mcal}$ for
all $a \in \ascr$, then clearly $\mcal
\subseteq \mcal_{\chi}(\pi)$.

(ii) Let $f \in \mcal_{\chi}(\pi)$.
Then
    \begin{align*}
\pi(a)(A f) = A\pi(a)f = \chi(a) A f, \quad a \in
\ascr.
    \end{align*}
    Hence, $A \mcal_{\chi}(\pi)
\subseteq \mcal_{\chi}(\pi)$. Since
$A^*$ also commutes with $\pi$, we
deduce that $A^* \mcal_{\chi}(\pi)
\subseteq \mcal_{\chi}(\pi)$.
Therefore, $\mcal_{\chi}(\pi)$ reduces
$A$.

(iii) According to (ii), $P_{\chi} \in
\pi(\ascr)^{\prime\prime}$, where
$\pi(\ascr)^{\prime\prime}$ denotes the bicommutant of
$\pi(\ascr)$. Therefore, by the von Neumann double
commutant theorem (see \cite[Theorem~4.1.5]{Mur90}),
we have $P_{\chi} \in W^*(\pi)$.
   \end{proof}
The proof of the next lemma follows directly from
Definition~\ref{dsfim}.
   \begin{lemma} \label{2ryzl}
Let $\ascr$ be a unital $C^*$-algebra and $\chi$ be a
character of $\ascr$. Let $\hh$ and $\kk$ be Hilbert
spaces, $\pi\colon \ascr \to \ogr(\hh)$ and
$\rho\colon \ascr \to \ogr(\mathcal{K})$ be
representations and $R \in \ogr(\mathcal{H},
\mathcal{K})$ be a contraction. Suppose that the
triplet $(\pi, \rho, R)$ has an isometric
$\chi$-decomposition relative to the quadruple
$(\hh_0, \hh_1, \kk_0, \kk_1$). Then
   \begin{enumerate}
   \item[(i)] $\pi(a)  = R_0^*\rho(a)|_{\kk_0}
R_0 \oplus \chi(a) I_{\hh_1} = R^*\rho(a) R + \chi(a)
\triangle_R$ for all $a \in \ascr$,
   \item[(ii)] $\pi(a)|_{\hh_1}  = R_1^*\rho(a)|_{\kk_1}
R_1 + \chi(a) \triangle_{R_1} = \chi(a) I_{\hh_1}$ for
all $a \in \ascr$,
   \item[(iii)] $\rho(a)  = \rho(a)|_{\kk_0}
\oplus \chi(a) I_{\kk_1}$ for all $a \in \ascr$.
   \end{enumerate}
   \end{lemma}
   \begin{proof}[Proof of Theorem~\ref{rroiu}]
(ii) Immediate, because $\rho$ is a representation.

(i)\&(iii) Since, by Lemma~\ref{mcxi},
$\mcal_{\chi}(\pi)$ reduces $\pi$, and $\pi(e)=I$, we
deduce from Theorem~\ref{soorh}($\alpha$) that
   \begin{align*}
\lcal_{\chi} = \bigvee_{a \in \ascr}
\rho(a) R (\mcal_{\chi}(\pi)) =
\bigvee_{a \in \ascr} R
\pi(a)(\mcal_{\chi}(\pi)) =
\overline{R(\mcal_{\chi}(\pi))}.
   \end{align*}
Similarly, we show that $\lcal_{\chi_{\perp}} =
\overline{R(\mcal_{\chi}(\pi)^{\perp})}$.

Now take $f\in \mcal_{\chi}(\pi)$. Then it follows
that
   \begin{align*}
0=\chi(a) f = \pi(a) f \overset{\eqref{xzaw}}=
R^*\rho(a)R f + \chi(a) \triangle_{R} f = R^*\rho(a)R
f, \quad a\in \ker\chi.
   \end{align*}
Since $a^*a\in \ker\chi$ for every $a \in \ker\chi$,
we deduce that
   \begin{align*}
\|\rho(a)Rf\|^2 = \is{R^*\rho(a^*a)Rf}{f}=0, \quad a
\in \ker\chi.
   \end{align*}
This implies that
   \begin{align} \label{rery}
\rho(a)R|_{\mcal_{\chi}(\pi)} = 0, \quad a \in
\ker\chi.
   \end{align}

Next, we show that $\lcal_{\chi} \perp
\lcal_{\chi_{\perp}}$. Let $a,b \in \ker\chi$ and
$\alpha, \beta \in \cbb$. By Theorem~\ref{soorh-2}(i),
$R^*R$ commutes with $\pi$, so by
Lemma~\ref{mcxi}(ii), $\mcal_{\chi}(\pi)$ reduces
$R^*R$, i.e., $R^*R f \in \mcal_{\chi}(\pi)$ for every
$f\in \mcal_{\chi}(\pi)$. This yields (i) and
   \allowdisplaybreaks
   \begin{align*}
\is{\rho(a+\alpha e)R f}{\rho(b+\beta e)R g} &
\overset{\eqref{rery}} = \is{\alpha R f}{\rho(b)Rg
+\beta R g}
   \\
& \hspace{1ex}= \alpha \is{\rho(b^*)R f}{Rg} + \alpha
\bar \beta \is{R^*R f}{g}
   \\
& \overset{\eqref{rery}} = \alpha \bar \beta \is{R^*R
f}{g} =0, \quad f \in \mcal_{\chi}(\pi), \, g \in
\mcal_{\chi}(\pi)^{\perp}.
   \end{align*}
Since $\ascr=\ker\chi+\comp \cdot e$, we deduce that
$\lcal_{\chi} \perp \lcal_{\chi_{\perp}}$, and
consequently $\lcal_{\chi_\perp} \subseteq
\lcal_{\chi}^\perp$. This proves (iii).

(iv) By Theorems~\ref{delt} and~\ref{soorh}, together
with Lemma~\ref{mcxi}(i), we have
   \begin{align*}
\rho(a) \big(\rho(b) R f \big) & = \rho(ab) R f =
R\pi(ab)f = R \pi(a) \pi(b) f
   \\
& = \chi(a) R \pi(b) f = \chi(a) \big(\rho(b) Rf),
\quad f \in \mcal_{\chi}(\pi), \, a,b \in \ascr.
   \end{align*}
This implies that $\lcal_{\chi} \subseteq
\mcal_{\chi}(\rho)$, which, together with (iii), gives
the first inclusion in (iv). Since, by
Theorem~\ref{soorh-2},
   \begin{align*}
\pi (a) = \pi(a)|_{\nul(\triangle_{R})} \oplus \chi(a)
I_{\overline{\ran(\triangle_{R})}}, \quad a \in \ascr,
   \end{align*}
we infer from Lemma~\ref{mcxi}(i) that
$\ran(\triangle_{R}) \subseteq
\mcal_{\chi}(\pi)$. This proves (iv).

(v) That the operators $R_0$ and $R_1$
are well-defined, and that $R = R_0
\oplus R_1$, follows from (iii). By
(iv), we have
$\mcal_{\chi}(\pi)^{\perp} \subseteq
\nul(\triangle_{R})$, which implies
that $R_0$ is an isometry. By
Lemma~\ref{mcxi}, $\mcal_{\chi}(\pi)$
reduces $\pi$ and
$\pi(a)|_{\mcal_{\chi}(\pi)} =
\chi(a)I_{\mcal_{\chi}(\pi)}$ for every
$a \in \ascr$. According to (iii) and
\eqref{rery}, we have
$\rho(a)|_{\lcal_{\chi}}=0$ for every
$a \in \ker\chi$. Since
$\ascr=\ker\chi+\comp \cdot e$, it
follows that
$\rho(a)|_{\lcal_{\chi}}=\chi(a)
I_{\lcal_{\chi}}$ for every $a \in
\ascr$. Summarizing, conditions
(i)-(iv) of Definition~\ref{dsfim} are
satisfied. That condition (v) of
Definition~\ref{dsfim} is also
satisfied can be deduced from the
following identities:
   \begin{align*}
\pi(a) & \overset{\eqref{xzaw}}= R^*\rho(a)R + \chi(a)
\triangle_{R}
   \\
& = \left(R_0^*\rho(a)|_{\kk_0} R_0 + \chi(a)
\triangle_{R_0}\right) \oplus \left(R_1^*
\rho(a)|_{\kk_1} R_1 + \chi(a) \triangle_{R_1}\right)
   \\
& = R_0^*\rho(a)|_{\kk_0} R_0 \oplus \chi(a)
I_{\hh_1}, \quad a \in \ascr.
   \end{align*}

(vi) This follows from (v) and Theorems~ \ref{delt}
and~\ref{soorh}($\alpha$).

To prove the ``moreover'' part, suppose
that $\kk = \bigvee_{a \in \ascr}
\rho(a)R(\hh)$. Let $g \in \kk$ be such
that $g \perp \lcal_{\chi} \oplus
\lcal_{\chi_{\perp}}$. Then, we have
   \begin{align*}
\is{g}{\rho(a) R (h\oplus h')} = \is{g}{\rho(a) R h} +
\is{g}{\rho(a) R h'} = 0,
   \end{align*}
for all $h\in \mcal_{\chi}(\pi)$,
$h'\in \mcal_{\chi}(\pi)^{\perp}$ and
$a \in \ascr$. This implies that $g
\perp \kk$, and hence $g=0$.
Consequently, $\lcal_{\chi} \oplus
\lcal_{\chi_{\perp}} = \kk$. By (v),
$R_0$ is an isometry. In turn, by
(iii), we have
$\lcal_{\chi}^{\perp}=\lcal_{\chi_{\perp}}=
\overline{\ran(R_0)}$, which shows that
$R_0$ is unitary.
   \end{proof}
   \begin{proof}[Proof of Theorem~\ref{rconrep}]
   (i)$\Rightarrow$(ii)
If $R$ is isometric, then by Theorem~\ref{delt} the
implication \eqref{plyc} holds. If $R$ is not
isometric, then \eqref{xzaw} yields \eqref{plyc} as
well.

   (ii)$\Rightarrow$(i)
Assume $\pi\colon \ascr \to \ogr(\hh)$ is a
representation on a separable Hilbert space $\hh$, and
$\varPhi\colon \ascr \to \ogr(\hh)$ is a UCP map such
that $\pi(a)=\varPhi(a)$ for all $a\in G$. By the
Stinespring dilation theorem,
$\varPhi(a)=P\rho(a)|_{\hh}$ for all $a \in \ascr$,
where $\rho\colon \ascr \to \ogr(\kk)$ is a
representation on a Hilbert space $\kk$, $\hh
\subseteq \kk$ and $P\in \ogr(\kk)$ is the orthogonal
projection of $\kk$ onto $\hh$. We can assume that
$\kk$ is separable (see the proof of
Theorem~\ref{ogylu}). Applying (ii) to the triplet
$(\pi,\rho,R)$ with $Rh=h$ for $h\in \hh$, we~get
   \begin{align*}
\pi(a) = R^*\rho(a) R + \chi(a) \triangle_R =
P\rho(a)|_{\hh}= \varPhi(a), \quad a \in \ascr.
   \end{align*}
Hence, $\pi|_{G}$ has the unique extension property,
and consequently $G$ is hyperrigid.

(i)$\Rightarrow$(iv)\&(i)$\Rightarrow$(v)
Both implications are consequences of
Theorem~\ref{rroiu}.

(v)$\Rightarrow$(ii)
Let $(\pi, \rho, R)$ be as in (ii), and suppose that
$\pi(g) = R^* \rho(g) R$ for all $g \in G$. In view of
Lemma~\ref{deqa}, the triplet $(\pi,
\rho_{\mathrm{m}}, R_{\mathrm{m}})$ satisfies the
minimality condition and the if-clause of the
implication in (v) holds. Hence, $(\pi,
\rho_{\mathrm{m}}, R_{\mathrm{m}})$ has a unitary
$\chi$-decomposition. Since
$\triangle_R=\triangle_{R_{\mathrm{m}}}$ and, by
Lemma~\ref{deqa}, $R^*\rho(a) R =
R_{\mathrm{m}}^*\rho_{\mathrm{m}}(a) R_{\mathrm{m}}$
for all $a\in \ascr$, we deduce from
Lemma~\ref{2ryzl}(i) that the then-clause of (ii)
holds.

(iv)$\Rightarrow$(ii)
Argue as in the proof of (v)$\Rightarrow$(ii).

(ii)$\Rightarrow$(iii) Obvious.

(iii)$\Rightarrow$(ii) Apply the identity
$\ascr=\ker\chi+\comp \cdot e$.
   \end{proof}
   \begin{corollary} \label{tiply}
Let $G$ be a hyperrigid subset of a unital
$C^*$-algebra $\ascr$. Let $\pi\colon \ascr \to
\ogr(\hh)$ and $\rho\colon \ascr \to \ogr(\kk)$ be
representations on separable Hilbert spaces $\hh$ and
$\kk$ respectively, and $R\in \ogr(\hh,\kk)$ be a
non-isometric contraction such that $\pi(g) =
R^*\rho(g)R$ for every $g\in G$. Then the following
conditions are equivalent{\em :}
   \begin{enumerate}
   \item[(i)] $G$ is rigid at $0$,
   \item[(ii)] there is no character of $\ascr$ vanishing on $G$.
   \end{enumerate}
Moreover, there exists at most one state $\phi$ on
$\ascr$ such that $G \subseteq \ker \phi$, and if such
a state exists, it is a character of $\ascr$.
   \end{corollary}
   \begin{proof}
In view of Theorem~\ref{ogylu}, only the ``moreover''
part requires the proof. Since $R$ is not an isometry,
we see that $e\notin G$, where $e$ is the unit of
$\ascr$. Suppose that $\phi_1$ and $\phi_2$ are states
on $\ascr$ vanishing on $G$. From Theorem~\ref{ogylu},
it follows that $\phi_1$ and $\phi_2$ are characters
of $\ascr$ that coincide on the set $G$. Since $G$
generates the $C^*$-algebra $\ascr$, we conclude that
$\phi_1=\phi_2$.
   \end{proof}
   \begin{corollary} \label{hpse}
Let $G$ be a hyperrigid subset of a unital
$C^*$-algebra $\ascr$. Then the following conditions
are equivalent{\em :}
   \begin{enumerate}
   \item[(i)] there exists a character $\chi$ of $\ascr$
such that $G \subseteq \ker \chi$,
   \item[(ii)] there exist representations
$\pi\colon \ascr \to \ogr(\hh)$ and $\rho\colon \ascr
\to \ogr(\kk)$ on separable Hilbert spaces $\hh$ and
$\kk$, a non-isometric contraction $R\in
\ogr(\hh,\kk)$ and a state $\phi$ on $\ascr$ such that
$G \subseteq \ker \phi$ and $\pi(g) = R^*\rho(g)R$ for
all $g\in G$,
   \item[(iii)] there exist representations
$\pi\colon \ascr \to \ogr(\hh)$ and $\rho\colon \ascr
\to \ogr(\kk)$ on separable Hilbert spaces $\hh$ and
$\kk$, a pure contraction $R\in \ogr(\hh,\kk)$ with
dense range and a state $\phi$ on $\ascr$ such that $G
\subseteq \ker \phi$ and $\pi(g) = R^*\rho(g)R$ for
all $g\in G$.
   \end{enumerate}
Moreover, if there exists a character $\chi$ of
$\ascr$ vanishing on $G$, then any state on $\ascr$
vanishing on $G$ coincides with $\chi$.
   \end{corollary}
   \begin{proof}
(i)$\Rightarrow$(iii) Take nonzero Hilbert spaces
$\hh$ and $\kk$. Define the representations $\pi\colon
\ascr \to \ogr(\hh)$ and $\rho\colon \ascr \to
\ogr(\kk)$ by $\pi(a)=\chi(a)I_{\hh}$ and
$\rho(a)=\chi(a)I_{\kk}$ for every $a\in \ascr$. Set
$\phi=\chi$. Take an arbitrary $R\in \ogr(\hh,\kk)$.
Then $\pi(g)=R^*\rho(g)R$ for every $g\in G$. If, in
addition, $\dim \hh \Ge \dim \kk$, then $R$ can always
be chosen to be a pure contraction with dense range.

(iii)$\Rightarrow$(ii) Obvious.

(ii)$\Rightarrow$(i) This implication is a direct
consequence of Theorem~\ref{ogylu}.

The ``moreover'' part follows from the ``moreover''
part of Corollary~\ref{tiply} and implication
(i)$\Rightarrow$(ii).
   \end{proof}
   \begin{rem} As observed in the proof of
implication (i)$\Rightarrow$(iii) of
Corollary~\ref{hpse}, if there exists a character
$\chi$ of $\ascr$ such that $G \subseteq \ker \chi$,
then for all Hilbert spaces $\hh$ and $\kk$ and all
operators $R\in \ogr(\hh,\kk)$, there exist
representations $\pi\colon \ascr \to \ogr(\hh)$ and
$\rho\colon \ascr \to \ogr(\kk)$ such that $\pi(g) =
R^*\rho(g)R$ for all $g\in G$. \hfill $\diamondsuit$
   \end{rem}
For an isometry $V\in \ogr(\hh,\kk)$, the {\em defect
space} $\dfi(V)$ of $V$ is defined by
   \begin{align*}
\dfi(V)=\ran(V)^{\perp} = \nul(V^*).
   \end{align*}
It is worth noting that the term
``defect space'' is also used in the
context of contractions to refer to the
closure of the range of the defect
operator (see \cite[Sec.\
I.3]{Sz-F-B-K10}). These notions
coincide only when the isometry is
unitary. Since the second
interpretation of defect space does not
occur in this paper, we adapt the first
one.
   \begin{lemma} \label{rdrs}
Let $R\in \ogr(\hh,\kk)$ be a contraction.
Then
   \begin{enumerate}
   \item[(i)]
$R(\nul(\triangle_R))^\perp = (R^*)^{-1}
(\overline{\ran(\triangle_R)})$,
   \item[(ii)] $\overline{\ran(R \triangle_R)}^{\perp} =
(R^*)^{-1} \left( \nul(\triangle_R) \right)$.
   \end{enumerate}
   \end{lemma}
   \begin{proof}
(i) Let $g \in \kk$. Then $g\in
R(\nul(\triangle_R))^\perp$ if and only if
   \begin{align*}
0 = \langle g, R h \rangle = \langle R^* g,
h \rangle, \quad h \in \nul(\triangle_R),
   \end{align*}
which is equivalent to $R^*g \in
\nul(\triangle_R)^{\perp} =
\overline{\ran(\triangle_R)}$, and thus to $g\in
(R^*)^{-1} (\overline{\ran(\triangle_R)})$.

(ii) Let $g \in \kk$. Then $g \in
\overline{\ran(R \triangle_R)}^{\perp}$ if
and only if
   \begin{align*}
0 = \langle g, R \triangle_R h\rangle =
\langle R^* g, \triangle_R h \rangle, \quad
h \in \hh,
   \end{align*}
which is equivalent to $R^* g\in
\overline{\ran(\triangle_R)}^{\perp}=
\nul(\triangle_R)$, and so to $g\in
(R^*)^{-1}(\nul(\triangle_R))$. This completes the
proof.
   \end{proof}
   \begin{proof}[Proof of Corollary~\ref{uspc}]
Assertions (i)-(vii) follow from
Theorems~\ref{soorh-2} and \ref{rroiu}, together with
Lem\-ma~\ref{mcxi}, using basic properties of
orthogonal projections.

(viii) By Lemma~\ref{rdrs}, we have
   \allowdisplaybreaks
   \begin{align*}
\dfi(R_{I}) & = \overline{\ran(R \triangle_R)}^{\perp}
\ominus R(\nul(\triangle_R))
   \\
& = (R^*)^{-1}(\nul(\Delta_R)) \cap
(R^*)^{-1}(\overline{\ran(\Delta_R)})
   \\
& = (R^*)^{-1}\big(\nul(\Delta_R) \cap
\overline{\ran(\Delta_R)}\big) = (R^*)^{-1}(\{0\}) =
\nul(R^*).
   \end{align*}
It remains to show that $\dfi(R_{I}) = \dfi(R_{S})$,
which can be done as follows using (ii), (v) and (vi):
   \allowdisplaybreaks
   \begin{align*}
\dfi(R_{I}) & = \kk_{I} \ominus
\ran(R_{I}) = \left[\kk_{U} \oplus
\kk_{S}\right] \ominus
\left[\ran(R_{U}) \oplus
\ran(R_{S})\right] = \{0\} \oplus
\dfi(R_{S}).
   \end{align*}
This completes the proof.
   \end{proof}
   \section{\label{Sec.7}Proofs of
Propositions~\ref{dtdw}, \ref{dttr}, \ref{kridut},
\ref{dwct} and \ref{axer}.}
   We begin by showing that the space
$\mcal_{\chi}(\pi)$, defined in \eqref{mycj}, plays,
in some sense, the role of a generalized joint
eigenspace when the underlying $C^*$-algebra is
commutative (see Lemma~\ref{pynsp}). This
correspondence becomes especially clear when the
representation $\pi$ is induced by a normal $d$-tuple
(see Lemma~\ref{versyn}).

It follows from the Gelfand-Naimark
theorem (see
\cite[Theorem~11.18]{Rud73}) and
\cite[Proposition~4.5]{Paul02} that if
$\ascr$ is a commutative unital
$C^*$-algebra and $\pi\colon \ascr \to
\ogr(\hh)$ is a representation on a
Hilbert space $\hh$, then there exists
a unique regular spectral measure
$E_{\pi}\colon
\borel(\mathfrak{M}_{\ascr}) \to
\ogr(\hh)$ such that
   \begin{align*}
\pi(a) = \int_{\mathfrak{M}_{\ascr}}
\hat{a} \, \D E_{\pi}, \quad a \in
\ascr,
   \end{align*}
where $\mathfrak{M}_{\ascr}$ denotes
the maximal ideal space of $\ascr$ and
$\hat{a}\colon \mathfrak{M}_{\ascr} \to
\cbb$ is the Gelfand transform of $a\in
\ascr$. The regularity of $E_{\pi}$
means that $\langle
E_{\pi}(\cdot)f,f\rangle$ is a regular
complex measure for every $f\in \hh$.
We refer to $E_{\pi}$ as the {\em
spectral measure} of the representation
$\pi$. It is well known that Borel
spectral measures on Polish spaces are
automatically regular (see
\cite[Theorem~II.3.2]{Par67}).

For commutative unital $C^*$-algebras,
the space $\mcal_{\chi}(\pi)$ can be
described in terms of the spectral
measure of $\pi$ as follows.
   \begin{lemma} \label{pynsp} Let $\ascr$ be a
commutative unital $C^*$-algebra and $\chi$
be a character of $\ascr$. Let $\pi\colon
\ascr \to B(\hh)$ be a representation of
$\ascr$ on a Hilbert space $\hh$ and
$E_{\pi}\colon \borel(\mathfrak{M}_{\ascr})
\to \ogr(\hh)$ be the spectral measure of
$\pi$. Then $\mcal_{\chi}(\pi) =
\ran(E_\pi(\{\chi\}))$.
   \end{lemma}
   \begin{proof}
If $h \in \mathcal{H}$, then
   \begin{align*}
h \in \mcal_\chi(\pi) & \iff \| (\pi(a) - \chi(a)I)h
\|^2 = 0 \quad \forall a \in \ascr
   \\
& \iff \int_{\mathfrak{M}_{\ascr}} |
\hat{a}(\tau) - \chi(a) |^2 \, \langle
E_\pi(\D \tau) h, h \rangle = 0 \quad
\forall a \in \ascr
   \\
& \iff \mathrm{supp} \langle E_\pi(\cdot)h,
h \rangle \subseteq \bigcap_{a \in \ascr}
\big\{\tau \in \mathfrak{M}_{\ascr}\colon
\tau(a) = \chi(a) \big\}
   \\
& \iff \mathrm{supp} \langle E_\pi(\cdot)h, h \rangle
\subseteq \{ \chi \}
   \\
& \iff \langle E_\pi(\mathfrak{M}_{\ascr} \setminus
\{\chi\})h, h \rangle = 0
   \\
& \iff (I - E_\pi(\{\chi\})) h = 0
   \\
& \iff h \in \ran(E_\pi(\{\chi\})),
   \end{align*}
which completes the proof.
   \end{proof}
According to Theorem~\ref{rroiu}(iv),
we have $R(\mcal_{\chi}(\pi)) \subseteq
\mcal_{\chi}(\rho)$. We now show that
if $\kk = \bigvee_{a \in \ascr}
\rho(a)R(\hh)$ and the $C^*$-algebra
$\ascr$ is commutative, then the
closure of $R(\mcal_{\chi}(\pi))$
coincides with $\mcal_{\chi}(\rho)$.
   \begin{lemma}
Let $G$ be a hyperrigid subset of a
commutative unital $C^*$-algebra
$\ascr$ and $\chi$ be a character of
$\ascr$ such that $G \subseteq
\ker\chi$. Let $\pi\colon \ascr \to
\ogr(\hh)$ and $\rho\colon \ascr \to
\ogr(\kk)$ be representations of
$\ascr$ on separable Hilbert spaces
$\hh$ and $\kk$, respectively, and $R
\in \ogr(\mathcal{H},\kk)$ be a
contraction such~that $\pi(g) = R^*
\rho(g) R$ for every $g\in G$. Suppose
that $\kk = \bigvee_{a \in \ascr}
\rho(a)R(\hh)$. Then
   \begin{align}  \label{clyp}
\mcal_{\chi}(\rho) = \overline{R(\mcal_{\chi}(\pi))}.
   \end{align}
   \end{lemma}
   \begin{proof}
Denote by $E_{\pi}$ and $E_{\rho}$ the spectral
measures of the representations $\pi$ and $\rho$,
respectively. First, we show that
   \begin{align} \label{kmynm}
\kk = \bigvee_{\Delta \in
\borel(\mathfrak{M}_{\ascr})}
E_\mathcal{\rho}(\varDelta) R(\hh).
   \end{align}
Indeed, let $g \in \kk \ominus \bigvee_{\Delta \in
\borel(\mathfrak{M}_{\ascr})}
E_\mathcal{\rho}(\varDelta) R(\hh)$. Then the
(regular) complex measure $\langle E_{\rho}(\cdot) R
h, g \rangle$ vanishes for every $h \in \hh$, which
implies that
   \begin{align*}
\is{\rho(a) Rh}{g} =
\int_{\mathfrak{M}_{\ascr}}
\hat{a}(\tau) \langle E_{\rho} (\D
\tau) Rh, g \rangle = 0, \quad h \in
\hh, \, a \in \ascr.
   \end{align*}
Therefore, $g \perp \bigvee_{a \in \ascr} \rho(a)
R(\hh) = \mathcal{K}$, implying that $g = 0$. This
proves \eqref{kmynm}.

Next, we show that
   \begin{align} \label{prvp}
E_{\rho}(\varDelta) R = R E_{\pi}(\varDelta), \quad
\varDelta \in \borel(\mathfrak{M}_{\ascr}).
   \end{align}
Indeed, it follows from
Theorem~\ref{ogylu}(ii) that
$R\pi(a)=\rho(a)R$ for every $a \in
\ascr$. We then have
   \begin{align} \notag
\int_{\mathfrak{M}_{\ascr}} \hat{a}(\tau)
\langle E_{\rho} (\D \tau) Rh, g \rangle & =
\is{\rho(a) Rh}{g} = \is{\pi(a) h}{R^* g}
   \\ \label{prwsp}
& = \int_{\mathfrak{M}_{\ascr}}
\hat{a}(\tau) \langle E_{\pi} (\D \tau) h,
R^* g \rangle, \quad h \in \hh, \, g \in
\kk.
   \end{align}
Since, by the Gelfand-Naimark theorem,
$C(\mathfrak{M}_{\ascr})= \{\hat{a}\colon a \in
\ascr\}$, we infer from \eqref{prwsp} and the Riesz
representation theorem (see
\cite[Theorem~6.19]{Rud87}) that $\langle E_{\rho}
(\cdot) Rh, g \rangle = \langle E_{\pi} (\cdot) h, R^*
g \rangle$ for all $h \in \hh$ and $g \in \kk,$ which
proves~\eqref{prvp}.

Now, by Lemma~\ref{pynsp}, we have
   \begin{align*}
\mcal_{\chi}(\rho) & = E_{\rho}(\{\chi\})
\kk \overset{\eqref{kmynm}} =
E_{\rho}(\{\chi\}) \bigvee_{\varDelta \in
\borel(\mathfrak{M}_{\ascr})}
E_{\rho}(\varDelta) R(\hh)
   \\
& \subseteq \bigvee_{\varDelta \in
\borel(\mathfrak{M}_{\ascr})}
E_{\rho}(\{\chi\}) E_{\rho}(\varDelta)
R(\hh)
   \\
& = \bigvee_{\varDelta \in
\borel(\mathfrak{M}_{\ascr})} E_{\rho}(\{\chi\} \cap
\varDelta) R(\hh) = \overline{E_{\rho}(\{\chi\})
R(\hh)}
   \\
& \hspace{-1ex}\overset{\eqref{prvp}}= \overline{R
E_{\pi}(\{\chi\}) \hh} =
\overline{R(\mcal_{\chi}(\pi))}.
   \end{align*}
Since, by Theorem~\ref{rroiu}(iv),
$\overline{R(\mcal_{\chi}(\pi))}
\subseteq \mcal_{\chi}(\rho)$, we
conclude that identity \eqref{clyp}
holds.
   \end{proof}
For representations $\pi \colon C(X)
\to \ogr(\hh)$ induced by normal
$d$-tuples $\tbold \in \ogr(\hh)^d$
with Taylor spectrum contained in $X$,
the space $\mcal_{\chi}(\pi)$ is the
joint eigenspace corresponding to the
joint eigenvalue $\lambdab$ of
$\tbold$, where $\lambdab$ is uniquely
determined by the character $\chi$.
Recall that the characters of $C(X)$
are precisely the evaluation
functionals at points of $X$.
    \begin{lemma} \label{versyn}
Suppose that $X$ is a nonempty compact
subset of $\comp^d$ $($$d\in \natu$$)$,
$\hh$ is a Hilbert space, $\tbold=(T_1,
\ldots, T_d) \in \ogr(\hh)^d$ is a normal
$d$-tuple such that $\sigma(\tbold)
\subseteq X$, and $\lambdab = (\lambda_1,
\ldots, \lambda_d) \in X$. Then
   \begin{align} \label{cegnv}
\mcal_{\chi_{_{\lambdab}}}(\pi) =
\bigcap_{j=1}^d \nul(\lambda_j I-T_j),
   \end{align}
where $\pi$ is the representation of
$C(X)$ induced by $\tbold$ $($see
\eqref{dyfq}$)$ and
$\chi_{_{\lambdab}}$ is the character
of $C(X)$ defined by
$\chi_{_{\lambdab}}(f) = f(\lambdab)$
for $f\in C(X)$.
   \end{lemma}
   \begin{proof}
Set $\mcal=
\mcal_{\chi_{_{\lambdab}}}(\pi)$. Let
$E_{\tbold}$ be the joint spectral
measure of $\tbold$. By \eqref{dzer},
we have $\supp(E_{\tbold}) =
\sigma(\tbold) \subseteq X$. First, we
prove the inclusion ``$\subseteq$'' in
\eqref{cegnv}. To this end, take $h\in
\mcal$. Then $(f-f(\lambdab)
\boldsymbol{1})(\tbold)h=0$ for all
$f\in C(X)$, or equivalently, if and
only if $\int_{X}
|f(\zbold)-f(\lambdab)|^2
\is{E_{\tbold}(\D \zbold)h}{h} = 0$ for
all $f\in C(X)$. Substituting the
function $f_{\lambdab}\in C(X)$ defined
by
$f_{\lambdab}(\zbold)=\|\zbold-\lambdab\|$
for $\zbold\in X$, we get
   \begin{align*}
0=\int_{X} \|\zbold-\lambdab\|^2
\is{E_{\tbold}(\D \zbold)h}{h} & =
\sum_{j=1}^d \int_{\cbb^d}
|\xi_j(\zbold)-\lambda_j|^2
\is{E_{\tbold}(\D \zbold)h}{h}
   \\
& \hspace{-1ex}\overset{\eqref{tyjyt}}=
\sum_{j=1}^d \|T_jh -\lambda_j h\|^2,
   \end{align*}
where $\xi_j$ is as in \eqref{corge}.
Consequently, $h \in \bigcap_{j=1}^d
\nul(\lambda_j I-T_j)$. Conversely, if
$h \in \bigcap_{j=1}^d \nul(\lambda_j
I-T_j)$, then arguing as above we see
that $\int_{X} \|\zbold-\lambdab\|^2
\is{E_{\tbold}(\D \zbold)h}{h}=0$,
which implies that $\supp
\is{E_{\tbold}(\cdot)h}{h} \subseteq
\{\lambdab\}$ (equivalently: $h \in
\ran(E_{\tbold}(\{\lambdab\}))$), and
thus
   \begin{align*}
\|(f-f(\lambdab)
\boldsymbol{1})(\tbold)h\|^2 = \int_{X}
|f(\zbold)-f(\lambdab)|^2 \is{E_{\tbold}(\D
\zbold)h}{h} = 0, \quad f\in C(X).
   \end{align*}
This implies that $h\in \mcal$. It is worth noting
that, along the way, we have proved that $\mcal =
\ran(E_{\tbold}(\{\lambdab\}))$.
   \end{proof}
   \begin{proof}[Proof of Proposition~\ref{kridut}]
It follows from the Stone-Weierstrass
theorem and the automatic continuity of
representations of $C^*$-algebras that,
if $X$ is a compact subset of $\cbb^d$
and both $\pi \colon C(X) \to
\ogr(\hh)$ and $\rho \colon C(X) \to
\ogr(\kk)$ are representations, then
for every $f \in C(X)$, there exists a
sequence $\{p_n\}_{n=1}^{\infty}$ of
complex polynomials in the
indeterminates $\xi_1, \ldots, \xi_d$
and $\bar\xi_1, \ldots, \bar\xi_d$ such
that
   \begin{align*}
\pi(f) & = \lim_{n \to \infty} p_n(T_1, \ldots, T_d,
T_1^*, \ldots, T_d^*),
   \\
\rho(f) & = \lim_{n\to\infty} p_n(N_1, \ldots, N_d,
N_1^*, \ldots, N_d^*),
   \end{align*}
where $T_j = \pi(\xi_j|_{X})$ and $N_j =
\rho(\xi_j|_{X})$ for $j=1, \ldots, d$. Using this
fact and Proposition~\ref{ryled}, one can easily
verify that condition (ii) (respectively, (iii)) of
Proposition~\ref{kridut} is equivalent to condition
(ii) (respectively, (v)) of Theorem~\ref{rconrep} with
$\ascr = C(X)$ and $\chi(f) = f(\lambdab)$ for $f \in
C(X)$. This theorem therefore completes the proof.
   \end{proof}
   \begin{proof}[Proof of Proposition~\ref{dtdw}]
This corollary follows from Propositions~\ref{gfza}
and \ref{ryled}, and Lemma~\ref{versyn}, using the
fact that if $\pi\colon C(X) \to \ogr(\hh)$ and
$\rho\colon C(X) \to \ogr(\kk)$ are representations
induced by normal $d$-tuples $\tbold$ and $\nbold$,
respectively, with Taylor spectrum in $X$, and if
$R\in \ogr(\hh,\kk)$ is a contraction, then
   \begin{align} \label{fdea}
   \begin{minipage}{70ex}
{\em $RT_j=N_jR$ for all $j=1,\ldots, d$ if and only
if $R\pi(a)=\rho(a)R$ for all $a\in \ascr$ with
$\ascr=C(X)$.}
   \end{minipage}
   \end{align}
The equivalence \eqref{fdea}, in turn, follows from
the Putnam-Fuglede theorem and the Stone-Weierstrass
theorem (as in the proof of Proposition~\ref{kridut}).
   \end{proof}
   \begin{proof}[Proof of
Proposition~\ref{dttr}]
      Argue similarly to the proof of
Proposition~\ref{dtdw}, using Proposition~\ref{jger}
in place of Proposition~\ref{gfza}, via the
equivalence \eqref{fdea}.
   \end{proof}
   \begin{proof}[Proof of
Proposition~\ref{dwct}]
   Apply Theorem~\ref{ogylu}, Proposition~\ref{ryled},
and the equivalence~\eqref{fdea}.
   \end{proof}
   \begin{proof}[Proof of Proposition~\ref{axer}]
(i)$\Leftrightarrow$(ii) According to
Proposition~\ref{ryled}, condition (ii) of
Theorem~\ref{ogylu} with $\ascr = C(X)$ is equivalent
to condition (ii) of Proposition~\ref{axer}. Hence,
defining the state $\phi$ on $C(X)$ by $\phi(f) =
\int_X f \, \D\mu$ for $f\in C(X)$, we deduce from
Theorem~\ref{ogylu} that the required equivalence
holds.

(ii)$\Rightarrow$(iii) Apply (ii) to the triplet
$(\kk,R,N)$ defined by $\kk:=\kk_1 \oplus \cdots
\oplus \kk_n$, $Rh:= R_1h \oplus \cdots \oplus R_n h$
for $h \in \hh$, and $N:=N_1 \oplus \cdots \oplus
N_n$.

(iii)$\Rightarrow$(ii) Set $\kk_i=\kk$ and $N_i=N$ for
$i=1, \ldots, n$, $R_1=R$ and $R_i=0$ for $i=2,
\ldots, n$. Substituting these into (iii) yields
$NR=RT$.

The ``moreover'' part follows from the above
arguments, together with Proposition~\ref{ryled},
Theorem~\ref{soorh}, and the equality
$\mathfrak{M}_{C(X)}=X$.
   \end{proof}
   \section{\label{9.e}Examples}
In this section we provide two examples showing that a
hyperrigid set may not be annihilated by any state,
that the corresponding intertwining relation does not
hold, and that the set $J$ in \eqref{zbyr} need not be
an ideal (cf.~Proposition~\ref{glmp}).
   \begin{exa} \label{nrhu}
Fix $\kappa \in \natu \cup \{\infty\}$ and set
$L_{\kappa} = \{j \in \natu \colon j \le \kappa\}$.
Let $\hh$ and $\kk$ be Hilbert spaces with $\dim \hh =
\dim \kk = \kappa$, and let $\{e_j\}$ and $\{f_j\}$ be
orthonormal bases of $\hh$ and $\kk$, respectively.
Take $\{r_j\}_{j =1}^{\kappa} \subseteq (0,1)$ with
$\inf_{j \in L_{\kappa}} r_j > 0$, and fix $\lambda
\in \cbb \setminus \{0\}$. For $j\in L_{\kappa}$, set
$t_j = \lambda (1 + r_j)$ and $n_j = \frac{\lambda (1
+ r_j)}{r_j}$. The sequences $\{t_j\}_{j=1}^{\kappa}$,
$\{n_j\}_{j=1}^{\kappa}$ and $\{r_j\}_{j=1}^{\kappa}$
are bounded. Consider the normal operators $T \in
B(\hh)$ and $N \in B(\kk)$, and the contraction $R \in
B(\hh,\kk)$ given by $T e_j = t_j e_j$, $N f_j = n_j
f_j$ and $R e_j = r_j f_j$ for $j \in L_{\kappa}$.
Define the compact set $X =
\overline{\{t_j\}_{j=1}^{\kappa} \cup
\{n_j\}_{j=1}^{\kappa}}$. Then $\sigma(T) \cup
\sigma(N) \subseteq X$, $0 \notin X$ and $\lambda
\notin X$. By \cite[Corollary~1, p.~795]{Sh20} (or
\cite[Theorem~2.4(i)]{P-S-24}), the set $G := \{\xi,
\bar{\xi}\xi\}$ is hyperrigid in $C(X)$, where
$\xi(z)=z$ for $z\in X$. The unique extension property
implies that $G_{\lambda}:=\{\xi - \lambda,
\bar{\xi}\xi\}$ is also hyperrigid. By construction,
$T - \lambda I = R^* (N - \lambda I)R$ and $T^*T =
R^*N^*N R$. Using Propositions~\ref{glmp}
and~\ref{ryled}, we deduce that $RT \ne NR$, and the
set $J := \big\{ a \in C(X)\colon a(T) = R^* a(N) R
\big\}$ is not an ideal. Moreover, there is no state
$\phi$ on $C(X)$ vanishing on $G_{\lambda}$. Indeed,
otherwise there would exist a Borel probability
measure $\mu$ on $X$ such that $\int_{X} |z|^2
\D\mu(z) = \phi(\bar{\xi} \xi) = 0$. Since $|z|^2 > 0$
for all $z \in X$, we would obtain $0 = \mu(X)$,
contradicting $\mu(X) = 1$.
   \hfill $\diamondsuit$
   \end{exa}
   \begin{exa} \label{MScherer}
Let $\xi(z)=z$ on
$\tbb=\{z\in\cbb\colon |z|=1\}$. By
\cite[Corollary~2, p.~795]{Sh20} (or
\cite[Theorem~2.4(i)]{P-S-24}), the set
$G := \big\{\xi\big\}$, and
consequently $G_{2} :=
\big\{\xi-2\big\}$, is hyperrigid in
$C(\tbb)$. Define $T =
\left[\begin{smallmatrix}
-1 & 0 \\
0 & 1
\end{smallmatrix}\right]$,
$N = \left[\begin{smallmatrix}
-1 & 0 \\
0 & -1
\end{smallmatrix}\right]$
and $R = \left[\begin{smallmatrix}
1 & 0 \\
0 & 1/\sqrt{3}
\end{smallmatrix}\right]$.
Then $T$ and $N$ are unitary, and
$\|R\|\Le 1$. Clearly, $R^*(N - 2I)R =
T - 2I$. As in Example~\ref{nrhu} (or
by direct computation), $RT \ne NT$,
the corresponding set $J$ is not an
ideal, and $G_2$ is not annihilated by
any state on $C(\tbb)$.
   \hfill $\diamondsuit$
   \end{exa}
   \appendix
\section{Annihilation of hyperrigid sets.}
The question of the existence of a
hyperrigid set $G$ annihilated by a
given character of the $C^*$-algebra
$\ascr$ is completely answered below in
the case of a singly generated
commutative unital $C^*$-algebra
$\ascr$.
   \begin{prop} \label{chyr}
Let $\ascr$ be a commutative unital
$C^*$-algebra generated by a single
element $t \in \ascr$. Then, for every
character $\chi$ of $\ascr$, there
exists a hyperrigid set $G \subseteq
\ascr$ of cardinality at most two such
that $G \subseteq \ker \chi$. Moreover,
if $G$ is a hyperrigid subset of
$\ascr$, then there exists at most one
character $\chi$ of $\ascr$ such that
$G \subseteq \ker \chi$.
   \end{prop}
   \begin{proof}
By \cite[Theorem~11.19]{Rud73}, there
is no loss of generality in assuming
that $\ascr=C(X)$ and $t=\xi$, where
$X$ is a nonempty compact subset of
$\cbb$ and $\xi$ is the so-called {\em
coordinate function} on $X$ defined by
$\xi(z)=z$ for $z\in X$. Given
$\omega\in X$, we define the map
$\tau_{\omega}\colon X \to X-\omega$ by
$\tau_{\omega}(z)=z-\omega$ for $z\in
X$. Then $0\in \tau_{\omega}(X)$. Set
$\widetilde{G}_{\omega}=\{\xi_{\omega},
\bar\xi_{\omega} \xi_{\omega}\}$, where
$\xi_{\omega}$ is the coordinate
function on $\tau_{\omega}(X)$. By
\cite[Theorem~2.4(i)]{P-S-24},
$\widetilde{G}_{\omega}$ is a
hyperrigid set of generators of
$C(\tau_{\omega}(X))$ and $f(0)=0$ for
every $f\in \widetilde{G}_{\omega}$.
Define the composition map
$\pi_{\omega}\colon C(\tau_{\omega}(X))
\to C(X)$ by $\pi_{\omega} (f)=f\circ
\tau_{\omega}$ for $f\in
C(\tau_{\omega}(X))$. Then
$\pi_{\omega}$ is a unital
$C^*$-algebra isomorphism. As a
consequence,
$G_{\omega}:=\pi_{\omega}(\widetilde{G}_{\omega})$
is a hyperrigid subset of $C(X)$
(cf.~Definition~\ref{dyfnh}). It is
easy to see that $f(\omega)=0$ for
every $f\in G_{\omega}$, or
equivalently $G_{\omega} \subseteq \ker
\chi_{\omega}$, where $\chi_{\omega}$
is the character of $C(X)$ given by
$\chi_{\omega}(f)=f(\omega)$ for $f\in
C(X)$. Since characters of $C(X)$ are
of the form $\chi_{\omega}$ with
$\omega \in X$ (see \cite[p.\ 271,
Example~(a)]{Rud73}), this implies that
for every character $\chi$ of $C(X)$,
there exists a hyperrigid set $G$ of
generators of $C(X)$ such that $G
\subseteq \ker \chi$.

The ``moreover'' part follows from
Corollary~\ref{hpse}.
   \end{proof}
As shown below, the assumption
``$G\subseteq \ker \phi$'' appearing in
Theorem~\ref{soorh} can be replaced by
more general conditions.
   \begin{prop} \label{sweq}
Let $G$ be a hyperrigid subset of a
unital $C^*$-algebra $\ascr$, let
$\pi\colon \ascr \to \ogr(\hh)$ and
$\rho\colon \ascr \to \ogr(\kk)$ be
representations of $\ascr$ on separable
Hilbert spaces $\hh$ and $\kk$,
respectively, and let $R\in
\ogr(\hh,\kk)$ be a contraction such
that $\pi(g) = R^*\rho(g)R$ for all
$g\in G$. Then the following assertions
hold{\em :}
   \begin{enumerate}
   \item[(i)] If $\{\phi_{\iota}\}_{\iota \in
\varLambda}$ is a family of states on $\ascr$ such
that $G \subseteq \bigcap_{\iota \in \varLambda} \ker
\phi_{\iota}$, then $R \pi(a) = \rho(a) R$ for all
$a\in \ascr$ and $\ker \phi_{\iota} \subseteq J$ for
all $\iota \in \varLambda$, where $J$ $($see {\em
\eqref{zbyr}}$)$ is a closed ideal in $\ascr$.
   \item[(ii)] If $\{\phi_{\iota}\}_{\iota \in
\varLambda}$ is as in {\em (i)} and $R$
is non-isometric, then there exists a
character $\chi$ of $\ascr$ such that
$\phi_{\iota}=\chi$ for all $\iota\in
\varLambda$.
   \item[(iii)] If $\mcal$ is a
Hilbert space and $\varPsi\colon \ascr
\to \ogr(\mcal)$ is a nonzero
completely positive map such that $G
\subseteq \ker \varPsi$, then there
exists a character $\chi$ of $\ascr$
such that $\varPsi(a)=\chi(a)
\varPsi(e)$ for all $a\in \ascr$.
   \end{enumerate}
   \end{prop}
   \begin{proof} (i) It is a
direct consequence of Theorem~\ref{soorh}($\alpha$).

(ii) Suppose $R$ is non-isometric. Then condition
($\alpha$) of Theorem~\ref{soorh} implies that for
every $\iota\in \varLambda$, $\ker \phi_{\iota}=J$ and
$\phi_{\iota}$ is a character of $\ascr$. Hence,
$\phi_{\iota}=\chi$ for every $\iota\in \varLambda$,
where $\chi$ is a character of $\ascr$.

(iii) Suppose now that $G \subseteq
\ker \varPsi$ for some nonzero
completely positive map $\varPsi\colon
\ascr \to \ogr(\mcal)$. We will show
that
   \begin{align} \label{insq}
G \subseteq \bigcap_{f \in \varOmega}
\ker \phi_f,
   \end{align}
where $\varOmega=\big\{f\in \mcal\colon
\|f\|=1 \text{ and }
\is{\varPsi(e)f}{f} \neq 0\big\}$, and
$\phi_f$ is the state on $\ascr$
defined by
   \begin{align*}
\phi_f(a)=\frac{\is{\varPsi(a)f}{f}}{\is{\varPsi(e)f}{f}},
\quad a\in \ascr, \, f \in \varOmega.
   \end{align*}
That inclusion \eqref{insq} holds is an
immediate consequence of the identity
$\ker \varPsi = \bigcap_{f \in
\varOmega} \ker \phi_f$, whose proof
proceeds as follows. The inclusion
``$\subseteq$'' is obvious, while the
converse ``$\supseteq$'' can be deduced
from the Cauchy-Schwarz inequality:
   \begin{align} \label{fguyw}
|\is{\varPsi(a)f}{f}|^2 \Le
\is{\varPsi(a^*a)f}{f}
\is{\varPsi(e)f}{f}, \quad f\in \mcal,
\, a \in \ascr.
   \end{align}
It follows from (ii) and \eqref{insq}
that there exists a character $\chi$ of
$\ascr$ such that $\phi_f = \chi$ for
all $f \in \varOmega$. Fix $g\in
\varOmega$. Then
   \begin{align*}
\frac{\is{\varPsi(a)f}{f}}{\is{\varPsi(e)f}{f}}
= \phi_f(a) = \phi_g(a) =
\frac{\is{\varPsi(a)g}{g}}{\is{\varPsi(e)g}{g}},
\quad f \in \varOmega.
   \end{align*}
This, together with \eqref{fguyw},
implies that
   \begin{align*}
\is{\is{\varPsi(e)g}{g} \varPsi(a)f}{f}
= \is{\is{\varPsi(a)g}{g}
\varPsi(e)f}{f}, \quad f\in \mcal.
   \end{align*}
Hence, we have
   \begin{align*}
\is{\varPsi(e)g}{g} \varPsi(a) =
\is{\varPsi(a)g}{g} \varPsi(e), \quad a
\in \ascr,
   \end{align*}
so
   \begin{align*}
\varPsi(a) =
\frac{\is{\varPsi(a)g}{g}}{\is{\varPsi(e)g}{g}}
\varPsi(e) = \phi_g(a) \varPsi(e) =
\chi(a) \varPsi(e), \quad a \in \ascr.
   \end{align*}
This shows that $\varPsi(a)=\chi(a)
\varPsi(e)$ for all $a \in \ascr$.
   \end{proof}
   \section{\label{App.B}Representations versus normal $d$-tuples.}
Proposition~\ref{ryled} below describes
the relationship between
representations of $C(X)$ and normal
$d$-tuples. For the reader's
convenience, we include its proof.
   \begin{prop} \label{ryled}
Let $X$ be a nonempty compact subset of
$\comp^d$ $($$d\in \natu$$)$, and let
$\hh$ be a Hilbert space. If $\pi\colon
C(X)\to\ogr(\hh)$ is a representation
and $T_j=\pi(\xi_j|_{X})$ for
$j=1,\dots,d$, then
$\tbold=(T_j)_{j=1}^d$ is a normal
$d$-tuple such that $\sigma(\tbold)
\subseteq X$, and
   \begin{equation} \label{dyfq}
\pi(f)= f (\tbold), \quad f\in C(X).
   \end{equation}
Conversely, if $\tbold=(T_j)_{j=1}^d\in
\ogr(\hh)^d$ is a normal $d$-tuple such
that $\sigma(\tbold) \subseteq X$, then
the map $\pi\colon C(X)\to\ogr(\hh)$
defined by \eqref{dyfq}, is the unique
representation satisfying
$\pi(\xi_j|_{X})=T_j$ for
$j=1,\dots,d$. Moreover, if $\rho\colon
C(X) \to \ogr(\kk)$ is a representation
on a Hilbert space $\kk$ and $R\in
\ogr(\hh,\kk)$, then $R\pi(f)=\rho(f)R$
for all $f\in C(X)$ if and only if
$RT_j=S_jR$ for $j=1,\ldots,d$, where
$S_j=\rho(\xi_j|_{X})$ for
$j=1,\dots,d$.
   \end{prop}
   \begin{proof}
Let $\pi\colon C(X)\to\ogr(\hh)$ be a
representation, and set
$\tbold=(T_j)_{j=1}^d$, where
$T_j=\pi(\xi_j|_{X})$ for
$j=1,\dots,d$. Then, clearly, $\tbold$
is a normal $d$-tuple. We now show that
the Taylor spectrum of $\tbold$ is
contained in $X$. To see this, observe
that
   \begin{align*}
\pi(\xi_j|_{X}) = T_j
\overset{\eqref{tyjyt}}= \int_{\cbb^d}
\xi_j \D E_{\tbold}
\overset{\eqref{dzer}}=
\int_{\sigma(\tbold)} \xi_j \D
E_{\tbold}, \quad j=1, \ldots, d.
   \end{align*}
This implies that for every complex
polynomial $p$ in the variables $\xi_1,
\ldots, \xi_d$ and $\bar\xi_1, \ldots,
\bar \xi_d$, we have
   \begin{align*}
\pi(p|_{X}) = \int_{\sigma(\tbold)} p
\, \D E_{\tbold}.
   \end{align*}
Hence, by the Stone-Weierstrass
theorem, we obtain
   \begin{align} \label{dqal}
\pi(f|_{X}) = \int_{\sigma(\tbold)} f\,
\D E_{\tbold}, \quad f \in C(X \cup
\sigma(\tbold)).
   \end{align}
Now, take $\lambdab \in \cbb^d\setminus
X$. Then, by Urysohn's lemma, there
exists an open set $V\subseteq \cbb^d$
and a continuous function $f_0 \colon
\cbb^d \to [0,\infty)$ such that
$f_0(\lambdab)=1$, $\lambdab \in V$,
$V\cap X=\emptyset$, and $\overline{\{z
\in \cbb^d\colon f_0(z) > 0\}}
\subseteq V$. Since $f_0|_{X}=0$, we
infer from \eqref{dzer} and
\eqref{dqal} that $E_{\tbold}(\{z\in
\cbb^d\colon f_0(z) > 0\})=0$, or
equivalently, that $\supp E_{\tbold}
\subseteq f_0^{-1}(\{0\})$. But since
$f_0(\lambdab)>0$, it follows from
\eqref{dzer} that $\lambdab \in \cbb^d
\setminus \sigma(\tbold)$. This shows
that $\sigma(\tbold) \subseteq X$.

Our next goal is to prove that
\eqref{dyfq} is valid. Using
\eqref{tyjyt} and \eqref{dzer}, we see
that equality \eqref{dyfq} holds for
$f=\xi_j|_{X}$ for each $j=1, \ldots,
d$. Since, by the Stone-Weierstrass
theorem, the set $\{\xi_1|_{X}, \ldots,
\xi_d|_{X}\}$ generates the
$C^*$-algebra $C(X)$, and since the
right-hand side of \eqref{dyfq} defines
a representation of $C(X)$, we conclude
that equality \eqref{dyfq} holds for
all $f\in C(X)$.

Now we prove the ``converse''
implication. It is well-known that the
Stone-von Neumann calculus
$C(\sigma(\tbold)) \in f \longmapsto
f(\tbold) \in \ogr(\hh)$ is a
representation, and consequently so is
$\pi$ defined by \eqref{dyfq}.
Moreover, we have
   \begin{align*}
\pi(\xi_j|_{X}) = \int_{\sigma(\tbold)}
\xi_j \, \D E_{\tbold} \overset{
\eqref{dzer}}= \int_{\cbb^d} \xi_j \,
\D E_{\tbold} \overset{\eqref{tyjyt}}=
T_j, \quad j = 1, \dots, d.
   \end{align*}
Arguing as in the previous paragraph,
we establish the uniqueness of the
representation $\pi$ satisfying
$\pi(\xi_j|_{X})=T_j$ for
$j=1,\dots,d$.

The ``moreover'' part follows from the
Stone-Weierstrass and Fuglede-Putnam
theorems. This completes the proof.
   \end{proof}
If $\pi$ and $\tbold$ are related as in
Proposition~\ref{ryled}, we say that
the representation $\pi$ is {\em
induced} by the normal $d$-tuple
$\tbold$.
   \subsection*{Acknowledgements}
The authors would like to thank Marcel Scherer for
inspiring discussions around Proposition~\ref{glmp}
and for providing Example~\ref{MScherer}.

   \bibliographystyle{amsalpha}

\begin{thebibliography}{99}
   \bibitem{Arv11}
W. Arveson, The noncommutative Choquet boundary II:
hyperrigidity, {\em Israel J. Math.} 184(2011),
349-385.
   \bibitem{Bi24} B. Bilich, Maximality of correspondence
representations, {\em Adv. Math.} {\bf 462} (2025)
110097.
   \bibitem{Bi-Dor24} B. Bilich,  A. Dor-On,
Arveson's hyperrigidity conjecture is false, arXiv:
2404.05018.
   \bibitem{Brown16} L. G. Brown, Convergence of functions of
self-adjoint operators and applications, {\em Publ.
Mat.} {\bf 60} (2016), 551--564.
   \bibitem{C-J-J-S} S. Chavan, Z. J. Jab{\l}o\'{n}ski, I. B. Jung, J.
Stochel, Taylor spectrum approach to Brownian-type
operators with quasinormal entry, {\em Ann. Mat. Pur.
Appl.} {\bf 200} (2021), 881-922,
   \bibitem{Cho-74} M. D. Choi, A Schwarz inequality
for positive linear maps on $C^*$-algebras, Illinois
J. Math. {\bf 18} (1974), 565-574.
  \bibitem{Clo18}
R. Clou\^{a}tre, Unperforated pairs of operator spaces
and hyperrigidity of operator systems,{\em Canad. J.
Math} {\bf 70} (2018), 1236-1260.
  \bibitem{Clo18b} R. Clou\^{a}tre, Non-commutative
peaking phenomena and a local version of the
hyperrigidity conjecture {\em Proc. Lond. Math. Soc.}
{\bf 117} (2018), 221--245.
   \bibitem{CH18}
R. Clou\^{a}tre, M. Hartz, Multiplier algebras of
complete Nevanlinna-Pick spaces: dilations, boundary
representations and hyperrigidity, {\em J. Funct.
Anal.} {\bf 274} (2018), 1690-1738.
   \bibitem{Cl-Ti21} R. Clou\^{a}tre, E. J. Timko,
Gelfand transforms and boundary representations of
complete Nevanlinna-Pick quotients, {\em Trans. Amer.
Math. Soc.} {\bf 374} (2021), 2107--2147.
   \bibitem{CH-Th24} R. Clou\^{a}tre, I. Thompson,
Rigidity of operator systems: tight extensions and
noncommutative measurable structures, arXiv:2406.16806
   \bibitem{Cun77} J. Cuntz,  Simple $C^*$-algebras
generated by isometries, {\em Comm. Math. Phys.} {\bf
57} (1977), 173-185.
   \bibitem{Dav96} K. R. Davidson,
$C^*$-algebras by example, Fields Inst.
Monogr., 6 American Mathematical
Society, Providence, RI, 1996. xiv+309
pp.
   \bibitem{DK21}
K. R. Davidson, M. Kennedy, Choquet order and
hyperrigidity for function systems, {\em Adv. Math.}
385 (2021), 107774.
   \bibitem{De-Ka-Pa25} J. A. Dessi, E. T. A. Kakariadis,
I. A. Paraskevas, A note on Arveson's hyperrigidity
and non-degenerate $C^*$-correspondences,
arXiv:2503.16618.
   \bibitem{Doug72} R. G. Douglas, {\em Banach
algebra techniques in operator theory}, Pure and
Applied Mathematics, Vol. 49, Academic Press, New
York-London, 1972.
   \bibitem{Ha-Pe82} F. Hansen, G. K. Pedersen, Jensen's inequality for
operators and L\"{o}wner's theorem, {\em Math. Ann.}
{\bf 258} (1982) 229--241.
   \bibitem{Har19} S. J. Harris, S.-J. Kim,
Crossed products of operator systems, {\em J. Funct.
Anal.} {\bf 276} (2019), 2156--2193.
   \bibitem{Kad-Rin83} R. V. Kadison, J. R. Ringrose,
{\em Fundamentals of the theory of operator algebras},
Vol. I, {\em Pure Appl. Math.}, 100, Academic Press,
Inc., New York, 1983, xv+398 pp.
   \bibitem{Ka-Sh19} E. T. A. Kakariadis, O. M. Shalit,
Operator algebras of monomial ideals in noncommuting
variables, {\em J. Math. Anal. Appl.} {\bf 472}
(2019), 738--813.
   \bibitem{Ka83}  W. E. Kaufman, Closed operators
and pure contractions in Hilbert space, {\em Proc.
Amer. Math. Soc.} {\bf 87} (1983), 83--87.
   \bibitem{KS15}
M. Kennedy, O. M. Shalit, Essential normality,
essential norms and hyperrigidity, {\em J. Funct.
Anal.} {\bf 268} (2015), 2990-3016.
   \bibitem{Kle14}
C. Kleski, Korovkin-type properties for completely
positive maps, {\em Illinois J. Math.} {\bf 58}
(2014), 1107-1116.
   \bibitem{MuhSol98}
P. S. Muhly, B. Solel, Tensor algebras
over $C^*$-correspondences:
representations, dilations and
$C^*$-envelopes, {\em J. Funct. Anal}.
\textbf{158} (1998), 389--457.
   \bibitem{Mur90}
G. J. Murphy, {\em $C^*$-algebras and operator
theory}, Academic Press, Inc., Boston, MA, 1990.
   \bibitem{Par67}
K. R. Parthasarathy, {\em Probability
measures on metric spaces}, Probability and
Mathematical Statistics, No. 3, Academic
Press, Inc., New York-London 1967.
   \bibitem{Paul02}
V. I. Paulsen, {\em Completely bounded maps and
operator algebras}, Cambridge Studies in Advanced
Mathematics, 78, Cambridge University Press,
Cambridge, 2002.
   \bibitem{P-S-roots23}
P. Pietrzycki, J. Stochel, On $n$th roots of bounded
and unbounded quasinormal operators, {\em Ann. Mat.
Pura. Appl.} {\bf 202} (2023), 1313-1333.
https://doi.org/10.1007/s10231-022-01281-z
   \bibitem{P-S-24}
P. Pietrzycki, J. Stochel, Hyperrigidity I: singly
generated commutative $C^*$-algebras, submitted
(2024).
   \bibitem{P-S25}
P. Pietrzycki, M. Scherer, J. Stochel,
Hyperrigidity III, arXiv:2501.04709.
   \bibitem{Rud73}
W. Rudin, {\em Functional analysis}, McGraw-Hill
Series in Higher Math., McGraw-Hill Book Co., New
York, 1973.
   \bibitem{Rud87}
W. Rudin, {\em Real and complex analysis},
McGraw-Hill, New York, 1987.
   \bibitem{Sal19} G. Salomon,  Hyperrigid subsets of
Cuntz-Krieger algebras and the property of rigidity at
zero, {\em J. Operator Theory} {\bf 81} (2019), 61-79.
   \bibitem{Sch24} M. Scherer, The hyperrigidity
conjecture for compact convex sets in $\rbb^2$,
arXiv:2411.11709
   \bibitem{Sch25}
M. Scherer, A finite-dimensional
counterexample for Arveson's
hyperrigidity conjecture, {\em Bull.
London Math. Soc.}. (2025)
https://doi.org/10.1112/blms.70120
   \bibitem{Sch12}
K. Schm\"{u}dgen, {\em Unbounded self-adjoint
operators on Hilbert space,} Graduate Texts in
Mathematics, 265, Springer, Dordrecht, 2012.
   \bibitem{Sh20} P. Shankar,  Hyperrigid generators
in $C^*$-algebras, {\em J. Anal.} {\bf 28} (2020),
791--797.
   \bibitem{Sti55}
W. F. Stinespring, Positive functions on
$C^*$-algebras, {\em Proc. Amer. Math. Soc.} {\bf 6}
(1955), 211--216.
   \bibitem{Sz-F-B-K10} B. Sz.-Nagy,  C. Foias, H.
Bercovici, L. K\'{e}rchy, {\em Harmonic analysis of
operators on Hilbert space}, Springer, New York, 2010.
   \bibitem{Weid80}
J. Weidmann, {\em Linear operators in Hilbert spaces},
Springer-Verlag, Berlin, Heidelberg, New York, 1980.
   \bibitem{Tay70}  J. L. Taylor, A joint
spectrum for several commuting operators, J.
Funct. Anal. 6 (1970), 172-191.
   \bibitem{Thom24} I. Thompson, An approximate
unique extension property for completely positive
maps, {\em J. Funct. Anal.} {\bf 286} (2024) 110193.
   \end{thebibliography}

   \end{document}